\documentclass[11pt]{amsart}
\usepackage[top=1 in,bottom=1 in, left=1 in, right= 1 in, includehead,includefoot]{geometry}

\usepackage{hyperref}

          \usepackage{amssymb}
          \usepackage{amsmath}
          \usepackage{amsthm}
          \usepackage{amsfonts}
          \usepackage{enumerate}
          \usepackage{url}
          \usepackage{hyperref}
          \usepackage{color}
          \usepackage{graphicx}

\usepackage[numbers,sort&compress]{natbib}

\newcommand{\comment}[1]{}
\def\fddto{\xrightarrow{\textit{f.d.d.}}}
\newcommand{\ind}{{\bf 1}}

\def\inddd#1{{\ind}_{\left\{#1\right\}}}

\newcommand{\proba}{\mathbb P}

\newcommand{\inv}{^{-1}}

\newcommand{\argmax}{\operatornamewithlimits{argmax}}

\newcommand{\eqnh}{\begin{eqnarray*}}
\newcommand{\eqne}{\end{eqnarray*}}
\newcommand{\eqnhn}{\begin{eqnarray}}
\newcommand{\eqnen}{\end{eqnarray}}
\newcommand{\equh}{\begin{equation}}
\newcommand{\eque}{\end{equation}}

\def\summ#1#2#3{\sum_{#1 = #2}^{#3}}
\def\prodd#1#2#3{\prod_{#1 = #2}^{#3}}

\newcommand{\eqd}{\stackrel{d}{=}}

\def\topp#1{^{(#1)}}

\def\abs#1{\left|#1\right|}
\def\sabs#1{|#1|}
\def\ccbb#1{\left\{#1\right\}}

\def\pp#1{\left(#1\right)}
\def\spp#1{(#1)}

\def\mmid{\;\middle\vert\;}

\def\vv#1{{\boldsymbol #1}}

\def\vvh{{\vv h}}
\def\vvc{{\vv c}}

\def\vvn{{\boldsymbol n}}

\def\vvb{{\vv b}}
\def\vvu{{\vv u}}
\def\vvv{{\vv v}}

\def\vvx{{\vv x}}
\def\vvz{{\vv z}}
\def\vvtau{{\vv \tau}}

\def\qmand{\quad\mbox{ and }\quad}

\def\qmwith{\quad\mbox{ with }\quad}
\def\mfa{\mbox{ for all }}

\def\mmas{\mbox{ as }}

\def\wt#1{\widetilde{#1}}
\def\wb#1{\overline{#1}}
\def\what#1{\widehat{#1}}

\def\limn{\lim_{n\to\infty}}

\def\weakto{\Rightarrow}

\def\R{{\mathbb R}}

\def\N{{\mathbb N}}

\def\B{{\mathbb B}}
\def\C{{\mathbb C}}

\def\calL{\mathcal L}

\def\calT{\mathcal T}

\def\H{{\mathsf H}}

\def\topp#1{^{(#1)}}

\def\d{{\rm d}}
\def\rmin{{\rm in}}
\def\rmout{{\rm out}}
\def\rmL{{\rm left}}
\def\rmR{{\rm right}}
\def\rmbr{{\rm br}}
\def\rmRe{{\rm Re}}

\newtheorem{Thm}{Theorem}[section]
\newtheorem{Lem}[Thm]{Lemma}

\newtheorem{Conj}[Thm]{Conjecture}

\theoremstyle{definition}
\newtheorem{Rem}[Thm]{Remark}

\numberwithin{equation}{section}

\def\step{{\rm step}}
\def\flat{{\rm flat}}
\def\GUE{{\rm GUE}}
\def\GOE{{\rm GOE}}
\def\dist{{\rm dist}}
\def\i{{\rm i}}
\def\Re{{\rmRe}}
\def\x{{\mathsf x}}
\def\h{{\mathsf h}}
\def\law{{\rm Law}}
\def\rx{{\mathsf x}}
\def\vvrx{{\vv\rx}}
\def\rh{{\mathsf h}}
\def\vvrh{{\vv\rh}}
\def\rf{\mathsf{f}}

\author{Zhipeng Liu}\address{Zhipeng Liu\\Department of Mathematics\\University of Kansas\\Lawrence, KS 66045, USA.}\email{zhipeng@ku.edu}

\author{Yizao Wang}\address{Yizao Wang\\Department of Mathematical Sciences\\University of Cincinnati\\2815 Commons Way\\Cincinnati, OH, 45221-0025, USA.}\email{yizao.wang@uc.edu}

\title[A conditional scaling limit of the KPZ fixed point]{A conditional scaling limit of the KPZ fixed point with height tending to infinity at one location}
\begin{document}\sloppy
\date{\today}
\begin{abstract}
We consider the asymptotic behavior of the KPZ fixed point $\{\H(x,t)\}_{x\in\R, t>0}$ conditioned on $\H(0,T)=L$ as $L$ goes to infinity.  The main result is a conditional limit theorem for the fluctuations of $\H$ in the region near the line segment
connecting the origin $(0,0)$ and $(0,T)$ for both step and flat initial conditions. The limit random field can be represented as a functional of two independent Brownian bridges, and in addition the limit random field depends also on the initial law of the KPZ fixed point. 
In particular for temporal fluctuations, the limit process indexed by line segment between $(0,0)$ and $(0,T)$, when the KPZ is with step initial condition, has the law of the minimum of two independent Brownian bridges; and when the KPZ is with flat initial condition the  limit process has the law of the minimum of two independent Brownian bridges, each in addition perturbed by a common Gaussian random variable. For spatial-temporal fluctuations, the conditional limit theorem sheds light on the asymptotic behaviors  of the point-to-point geodesic of the directed landscape conditioned on its length and as the length tends to infinity. 
\end{abstract}

\maketitle

\section{Introduction and main results}
\subsection{The model and main results}
The object of this paper is the so-called {\em Kardar--Parisi--Zhang (KPZ) fixed point}, a random two-dimensional field that arises from various models (random growth models, last passage percolations and directed polymers) that are loosely called the KPZ universality class (e.g.~\cite{baik99distribution,johansson00shape,johansson03discrete,borodin07fluctuations,tracy08integral,tracy09asymptotics,borodin14macdonald,matetski21KPZ,dauvergne18directed,quastel20convergence,Virag20heatandlandscape}). Throughout, we let $\{\H(x,t)\}_{x\in\R,t\ge 0}$ denote the KPZ fixed point, where $x$ and $t$ are the spatial and temporal parameters respectively. The field $\H$ also depends on the initial condition $\H(x,0)=\mathsf h_0(x)$ for some 
 function $\mathsf h_0$. The KPZ fixed point was conjectured to be the universal space-time limiting field of the height functions in all the models in the KPZ universality class,
 while its first characterization was obtained in \cite{matetski21KPZ} only recently. Therein, the authors defined $\H$ as the limit of the rescaled height function of an exactly solvable model in the KPZ universality class, the totally asymmetric simple exclusion process (TASEP), and then characterized $\H$ as the unique Markov process taking values in the space of upper-semicontinuous functions with explicit formula for transition probability.

In general, it is challenging to obtain explicit formulas for the distributions of $\H$. Such formulas are usually quite involved and obtained by exploiting the connections between certain exactly solvable models and the KPZ fixed point. Seminal works \cite{baik99distribution,johansson00shape} revealed the role of Tracy--Widom distributions for marginal (one-point) distributions. The finite-dimensional (multi-point) distribution  in the spatial direction (i.e., the law of $(\H(x_1,t),\dots,\H(x_d,t))$) for a fixed time point $t$) were obtained afterwards \cite{prahofer02scale,johansson03discrete,borodin07fluctuations,matetski21KPZ}. Most recently, the explicit formulas for joint cumulative distribution function of finite-dimensional distributions in general, 
with possibly different time points,
with were obtained in \cite{johansson17two} (two-dimensional) and \cite{johansson21multitime,liu22multipoint} (finite-dimensional). The latter two newly developed formulas in fact take quite different expressions, and a direct proof of their equivalence remains an interesting open question at this moment.

With explicit formulas at hand, it becomes now possible to carry out detailed analysis of path fluctuations of the KPZ fixed point. The main contribution of this paper is to characterize a local extremal behavior of the KPZ fixed point, using the recently developed formula \cite{liu22multipoint}.  More precisely, for some fixed $T>0$ we investigate the asymptotic joint behavior of $\{\H(x,t)\}_{x\in\R,t\in (0,T)}$ when the value of $\H(0,T)$ tends to infinity, established as convergence in distribution  with respect to the corresponding conditional law of $\H$.  Note that $\H$ depends on the initial condition $\H(x,0)=\mathsf h_0(x)$. We will focus on two specific initial conditions: the step initial condition $\mathsf h_0(x)=-\infty\inddd{x\ne 0}$, and the flat initial condition $\mathsf h_0(x)=0$. When needed, we will use $\H_\step$ and $\H_\flat$ to denote the KPZ fixed point with the step and flat initial conditions respectively.

Our main results are the following. We let $\fddto$ denote convergence of finite-dimensional distributions. 

 \begin{Thm}\label{thm:2}
 	For all $T>0$, as $L\to\infty$, 
 	\begin{multline}\label{eq:thm2}
 		\law\pp{\ccbb{\frac{\H(\frac{\rx T^{3/4}}{\sqrt 2 L^{1/4}},\tau T)-\tau \H(0,T)}{\sqrt 2 T^{1/4}L^{1/4}} }_{\rx\in\R,\tau\in(0,1)}\mmid  \H(0,T) = L}\\
 		\fddto \begin{cases}
 			\law\pp{\ccbb{\min\ccbb{\B_1^{\rmbr}(\tau)+\rx,  \B_2^{\rmbr}(\tau)-\rx}}_{\rx\in\R,\tau\in(0,1)}}, & \mbox{ if } \H = \H_\step,\\\\
 			\displaystyle \law\pp{\ccbb{\min\ccbb{\B_1^{\rmbr}(\tau)+\rx + \frac{1-\tau}{\sqrt 2} Z,  \B_2^{\rmbr}(\tau)-\rx - \frac{1-\tau}{\sqrt 2} Z}}_{\rx\in\R,\tau\in(0,1)}}, & \mbox{ if } \H = \H_\flat,
 		\end{cases}
 	\end{multline}
 	where on the right-hand side, $\{\B_1^\rmbr(\tau)\}_{\tau\in[0,1]}$ and $\{ \B_2^\rmbr(\tau)\}_{\tau\in[0,1]}$ denote two 
	i.i.d.~Brownian bridges over interval $[0,1]$, and  in the case with flat initial condition $Z$ is a standard normal random variable independent 
	from
	 $\B_1^\rmbr, \B_2^\rmbr$. 
 \end{Thm}
\begin{Rem}\label{rem:right}
An alternative interpretation of the limit random field is as follows. For fixed $\tau$, we describe the function $\mathfrak f_\tau(\rx) := \min\ccbb{\B_1^{\rmbr}(\tau)+\rx,  \B_2^{\rmbr}(\tau)-\rx}$ as a {\em shifted downward right angle}. It is convenient to describe this function via its graph $(\rx,\mathfrak f_\tau(\rx))_{\rx\in\R}\subset\R^2$. Introduce the function $\mathfrak f^*(\rx) = -|\rx|$ as the `downward right angle with vertex at $(0,0)$' (again think of its graph). Then, we can relate the graph of the two functions via
\equh\label{eq:shift}
\pp{\rx,\mathfrak f_\tau(\rx)}_{\rx\in\R} = \pp{\rx,\mathfrak f^*(\rx)}_{\rx\in\R} + \pp{\mathsf v_1(\tau),\mathsf v_2(\tau)}\equiv \pp{\rx+\mathsf v_1(\tau),\mathfrak f^*(\rx)+\mathsf v_2(\tau)}_{\rx\in\R},
\eque
where the graph of $\mathfrak f^*$ is shifted to the vertex at
\begin{equation}\label{eq:vertex}
(\mathsf v_1(\tau),\mathsf v_2(\tau)) =  		\left(\frac{-\B_1^{\rmbr}(\tau)+\B_2^{\rmbr}(\tau)}{2},\frac{\B_1^{\rmbr}(\tau)+\B_2^{\rmbr}(\tau)}{2}\right).
	\end{equation}
As a bivariate process, the vertex process $\{\mathsf v_1(\tau),\mathsf v_2(\tau)\}_{\tau\in(0,1)}$ given by \eqref{eq:vertex} has the law of  a bivariate Brownian bridge:  each coordinate process is an independent Brownian bridge scaled by $1/\sqrt 2$.

Then, we can view the limiting random field as a moving downwards right angle, with the vertex process following the law of a bivariate Brownian bridge for the case of step initial condition. For the case of flat initial condition the limit random field can be interpreted similarly: if we set  $\mathfrak f_\tau(\rx) = \min\{\B^\rmbr_1(\tau)+\rx+(1-\tau)Z/\sqrt 2,\B^\rmbr_2(\tau)-\rx - (1-\tau)Z/\sqrt 2\}$ accordingly, then \eqref{eq:shift} remains to hold with the same $\mathfrak f^*$ and now the vertex process becomes
\[
\pp{\mathsf v_1(\tau),\mathsf v_2(\tau)} = 
		\left(\frac{-\B_1^{\rmbr}(\tau)+\B_2^{\rmbr}(\tau)}{2} - \frac{1-\tau}{\sqrt 2} Z,\frac{\B_1^{\rmbr}(\tau)+\B_2^{\rmbr}(\tau)}{2}\right).
\]
\end{Rem}

\begin{Rem}\label{rem:conditional law}
We understand 
\[
\proba(\cdot\mid \H(0,T) = L) := \lim_{\epsilon\downarrow0}\proba(\cdot\mid\H(0,T) \in (L-\epsilon,L+\epsilon)), 
\]
and the proof is based on the formula
\begin{multline*}
\proba(\H(x_1,t_1)>h_1,\dots,\H(x_{m-1},t_{m-1})>h_{m-1}\mid \H(x_m,t_m) = h_m) \\
:= \frac{\partial}{\partial h_m}\proba(\H(x_1,t_1)>h_1,\dots,\H(x_{m},t_{m})>h_{m})/\frac{\partial}{\partial h_m}\proba(\H(x_{m},t_{m})>h_{m}).
\end{multline*}
The marginal distribution of 
$\H_\step$ ($\H_\flat$ respectively) is the GUE (GOE resp.)~Tracy--Widom distribution, 
and the formula of 
joint cumulative distribution function we use
was obtained in \citep{liu22multipoint}.
\end{Rem}

\begin{Rem}
Note that our limit theorem does not include the endpoints $\tau=0,1$. We do not include $\tau=0$ as the random field $\H_\step$ is not continuous at $\tau = 0+$, and in fact the formula we work with does not include $\tau = 0$. 

	 We do not include $\tau = 1$ in the present work, however, for a different reason. There is in fact a phase transition at $\tau=1$ for the conditional limit theorem of our interest. 
Indeed, in an upcoming work \cite{nissim22asymptotics}, the authors proved that for fixed $\tau>1$, the conditional limit distribution of $\H_\step(\rx,\tau T)$, properly normalized, given $\H_\step(0,T) = L$ as $L\to\infty$ is the GUE Tracy--Widom distribution, again by exploiting the formula from \cite{liu22multipoint}. 

By exploiting the same formula, we expect that there is a different scaling limit when $\tau=\tau_L$ depends on $L$ and $\tau_L\to 0$ (or $1$) as $L\to\infty$. We leave the two cases for investigations in the future.
\end{Rem}

\begin{Rem}
For the step initial condition, using the following invariance property of the KPZ fixed point (see the skew stationary property of Lemma 10.2 in \cite{dauvergne18directed})
\[
\ccbb{\H_\step\left(x,t\right)}_{x\in\R,t>0}\stackrel{d}= \ccbb{\H_\step\left(x+\frac{X}{T}t,t\right)-\frac{1}{t}\left(\left(x+\frac{X}{T}t\right)^2-x^2\right)}_{ x\in\R, t>0}
\]
for any fixed $X\in\R$ and $T>0$, we have
\begin{multline*}
\law\pp{\ccbb{\frac{\H_\step(\frac{\rx T^{3/4}}{\sqrt 2 L^{1/4}}+\tau X,\tau T)-\tau \H_\step(X,T)}{\sqrt 2 T^{1/4}L^{1/4}} }_{\rx\in\R,\tau\in(0,1)}\mmid  \H_\step(X,T) = L}\\
\quad\fddto \law\pp{\ccbb{\min\ccbb{\B_1^{\rmbr}(\tau)+\rx,  \B_2^{\rmbr}(\tau)-\rx}}_{\rx\in\R,\tau\in(0,1)}},
\end{multline*}
as $L\to\infty$. 

For the flat initial condition, since $\{\H_\flat(x,t)\}_{x\in\R,t>0}$ has the same law as $\{\H_\flat(x+c,t)\}_{x\in\R,t>0}$ for any fixed constant $c\in\R$, our theorem implies
\begin{multline*}
\law\pp{\ccbb{\frac{\H_\flat(\frac{\rx T^{3/4}}{\sqrt 2 L^{1/4}}+ X,\tau T)-\tau \H_\flat(X,T)}{\sqrt 2 T^{1/4}L^{1/4}} }_{\rx\in\R,\tau\in(0,1)}\mmid  \H_\flat(X,T) = L}\\
\fddto \law\pp{\ccbb{\min\ccbb{\B_1^{\rmbr}(\tau)+\rx + \frac{1-\tau}{\sqrt 2}\cdot Z,  \B_2^{\rmbr}(\tau)-\rx - \frac{1-\tau}{\sqrt 2}\cdot Z}}_{\rx\in\R,\tau\in(0,1)}},
\end{multline*}
as $L\to\infty$. 
\end{Rem}
\begin{Rem}
We are not aware of any similar conditional (second order) scaling limit theorems, but we mention two relevant results on  conditional limit theorems on the (first order) shape by proving certain concentration phenomena.  The first is the recent work by \cite{ganguly22sharp}. Their interest comes from a different aspect, and they proved various interesting results.  The most relevant result to our setup is \cite[Theorem 1.9]{ganguly22sharp}, where in our setup they considered $\tau=1$ and  proved
\[
\proba\pp{\sup_{x\in[-L^{1/2},L^{1/2}]}\frac{\abs{\H_\step(x,T)-L+2|x|L^{1/2}}}{L^{1/4}}> M_L\mmid \H_\step(0,T) = L}\le  \exp\pp{-C_1M_L^2} 
\]
with $M_L\le C_2L^{3/4}$ for some constants $C_1,C_2>0$. 
Note that they are interested in a different scaling  in the space parameter $x$, {\em of order $L^{1/2}$ away from the line segment} between $(0,0)$ and $(0,T)$ of out interest,  and also their result concerns the case $\tau=1$ that we exclude in this paper.

Another relevant and recent result is due to \cite{lamarre21kpz}, who identified the first order {\em space-time} limit shape of the so-called {\em KPZ equation in the weak noise regime}, also conditioning on the value at a fixed location being large and tending to infinite. The model considered therein is loosely related to but not exactly the KPZ fixed point investigated here, and also they are interested in the shape away from the line segment as in \cite{ganguly22sharp}.

In principle, our methodology can yield second order limit fluctuations at the different scaling order considered in the two papers above,
 away from the nearby region of the line segment between $(0,0)$ and $(0,T)$,
 although the analysis could be more involved. We leave this task for a future work.
\end{Rem}

We searched for a simple explanation on why the limit involves the minimum of two independent Brownian bridges, but without success yet. Nevertheless, our result is consistent with, and actually provides insight to, some recent discoveries and conjectures on point-to-point geodesics of directed landscape, as we explain now. Actually, the previous Remark \ref{rem:right} on the moving downward right angle is inspired by this connection: compare with Conjecture \ref{conj:03} below.

\subsection{Comments and conjectures regarding directed landscape}
We discuss a few conjectures suggested by our Theorem \ref{thm:2}, and the most interesting conjectures concern the recently introduced directed landscape and its geodesics, to be introduced in a moment. Originally, a main motivation of the present work was to better understand the local extrema behavior around the so-called point-to-point geodesic, when its length tends to infinity.

We shall start with a quick remark on Theorem \ref{thm:2}.  Our theorem is in the sense of convergence of the finite-dimensional distributions, and a natural question is to enhance the convergence to one in the space of $C((0,1))\times C((0,1))$ ($C$ denotes the space of continuous functions defined on the given interval). The only missing piece is the tightness of the normalized processes with respect to the conditional laws. We do not know how to prove this, but expect it to hold.

To simplify the notation we introduce 
\[
\wb \H_{\step,L}(\rx,\tau) := \frac{\H_\step\pp{\frac{\rx T^{3/4}}{ 2 L^{1/4}},\tau T} - \tau L}{T^{1/4}L^{1/4}}.
\]
(The reason for the slightly different normalization as in Theorem \ref{thm:2} will become clear in \eqref{eq:sqrt 2} below.)
Then, we are interested in the {\em argmax process}  and the {\em max process}
\[
\ccbb{\argmax_{\rx \in\R}\wb\H_{\step,L}(\rx,\tau)}_{\tau\in(0,1)} \qmand \ccbb{\max_{\rx \in\R}\wb\H_{\step,L}(\rx,\tau)}_{\tau\in(0,1)},
\]
respectively. The motivation of considering these two processes shall be clear, once we introduce the directed landscape and its geodesics.

Recall that the KPZ fixed point has continuous sample path when restricted to $t>0$. The definition of the argmax process is delicate and actually nontrivial, as explained in the next remark. 
\begin{Rem}\label{rem:unique}
It has been known that for every fixed $\tau\in(0,1)$, almost surely the argmax of $\wb\H_\step(\cdot, \tau)$ exists and is unique \cite{corwin14brownian,moreno13endpoint,pimentel14location}. 
However, it has been recently shown \cite{corwin21exceptional,dauvergne22non} that almost surely, there exists a non-empty fractal subset say $\calT$ of $(0,1)$ so that for $\tau\in \calT$, the maximum of $\wb\H_\step(\cdot,\tau)$ is not achieved uniquely. 
Therefore, while the maximum is always achieved for all $\tau\in(0,1)$ (so we can write $\max_{\rx\in \R}$ instead of $\sup_{\rx\in\R}$), to define the argmax process jointly in $\tau$ would require some work. 

Instead, to define the argmax process (without conditioning), one could first define it via its finite-dimensional distributions (for every choice of fixed $\tau_1,\dots,\tau_d$ and they are known to exist uniquely as aforementioned), which form a consistent family. Then, this family in turn determines the argmax process as a random element in $D((0,1))$, the space of c\`adl\`ag functions \cite{billingsley99convergence}. We let $\{\argmax_{\rx\in\R}\wb\H_{\step,L}(\rx,\tau)\}_{\tau\in(0,1)}$ denote the so-defined process.

Now, to define the argmax process with respect to the conditional law, one should also be careful regarding the uniqueness issue. In view of the discussions above and Remark \ref{rem:conditional law}, the conditional law of argmax process given $\H_\step(0,T) \in(L-\epsilon,L+\epsilon)$ is again well-defined (first for fixed $\tau_1,\dots,\tau_d$ and then they determine the law on $D((0,1))$), and therefore taking the limit $\epsilon\downarrow0$ we obtain the conditional law of the argmax process given $\H_\step(0,T) = L$. Strictly speaking, we need to show there is a well-defined limit when $\epsilon\downarrow 0$. We expect this to hold.

\end{Rem}

Then, the following conjecture can be proved if one shows 
the tightness for Theorem \ref{thm:2} with step initial condition.

\begin{Conj}
	\label{conj:03}
	For all $T>0$, as $L\to\infty$, 
\begin{multline}
	\label{eq:aux817}
		\law\pp{\ccbb{\argmax_{\rx\in\R}  \wb\H_{\step,L}(\rx,\tau),\max_{\rx\in\R}\wb\H_{\step,L}(\rx,\tau)}_{\tau\in(0,1)} \mmid \H_\step(0,T)=L}
\\		\to 
	\displaystyle\law\pp{\ccbb{\B_1^\rmbr(\tau),\B_2^\rmbr(\tau)}_{\tau\in(0,1)}},  
\end{multline}
where $\B_1^\rmbr$ and $\B_2^\rmbr$ denote two independent Brownian bridges, and
where the convergence is 
in distribution
in the space of $D((0,1),\R)\times D((0,1),\R)$.
\end{Conj}

Indeed, applying the continuous mapping theorem to Theorem \ref{thm:2}, it would follow that the limit of the left-hand side of \eqref{eq:aux817} is 
\begin{multline}\label{eq:sqrt 2}
\ccbb{\argmax_{\rx\in\R}\min\ccbb{\sqrt 2\B^\rmbr_1(\tau)+\rx,\sqrt 2\B^\rmbr_2(\tau)-\rx}, \max_{\rx\in\R}\min\ccbb{\sqrt 2\B^\rmbr_1(\tau)+\rx,\sqrt 2\B^\rmbr_2(\tau)-\rx}}_{\tau\in(0,1)}\\
\stackrel{a.s.}=\ccbb{\frac{\B_2^\rmbr(\tau)-\B_1^\rmbr(\tau)}{\sqrt 2},\frac{\B_1^\rmbr(\tau)+\B_2^\rmbr(\tau)}{\sqrt 2}}_{\tau\in(0,1)},
\end{multline}
which has the claimed joint distribution. It is remarkable that the argmax process and max process, two dependent objects, are asymptotically conditionally independent when $L\to\infty$.

\bigskip

Conjecture~\ref{conj:03} sheds light on the behavior of a closely related object, the geodesics of directed landscape, under a similar rare event of our consideration. To understand the extremal behavior of geodesics of directed landscape \cite{dauvergne18directed} was in fact our main motivation of the current work, following a recent result by one of us \cite{liu22when}. 

We first review some 
background
 exclusively regarding the directed landscape, denoted by $\calL$. We then explain how our result and the Conjecture \ref{conj:03} above are relevant to $\calL$. The directed landscape was introduced as the random field arising in limit theorems for Brownian last passage percolation \cite{dauvergne18directed}. Since then it has been proved to be the limiting field for several other KPZ models \cite{dauvergne21scaling}. 
A directed landscape is a four-parameter random field
\[
\ccbb{\mathcal{L}(x,s;y,t)}_{\R^4_\uparrow}\qmwith \R^4_\uparrow:=\ccbb{(x,s; y,t)\in \R^4: s<t},
\]
with continuous sample path.
Then, for any $s<t$, and continuous function $\pi\in C([s,t])$, one can define the length of $\pi$ with respect to the directed landscape $\mathcal{L}$ (see, for example \cite{rahman21infinite}) as
\[
\ell_{s,t}(\pi):=\inf_{n\in\N}\inf_{s=t_0<\cdots<t_n = t} \sum_{i=1}^n \mathcal{L}(\pi(t_{i-1}),t_{i-1};\pi(t_i),t_i), \qmwith \pi=\{\pi(r)\}_{r\in [s,t]}\in C([s,t]).
\]

The {\em geodesic of the directed landscape} $\calL$ between two fixed points $(x,s)$ to $(y,t)$, with $(x,s;y,t)\in \R^4_\uparrow$, is a continuous path $\pi\in C([s,t])$, with $\pi(s)=x$ and $\pi(t)=y$,  which has the maximal length $\ell_{s,t}(\pi)$ with respect to $\mathcal{L}$. 
It has been proved that the geodesic exists and is unique almost surely \cite{dauvergne18directed}. {For such a geodesic $\pi$, we refer to $\{\calL(x,s;\pi(r),r)\}_{r\in[s,t]}$ as the directed landscape {\em along the geodesic $\pi$} between $(x,s)$ and $(y,t)$.}

From now on, we let 
\equh\label{eq:pi*}
\pi^* := \argmax_{\substack{\pi\in C([0,T])\\
\pi(0) = 0, \pi(T) = 0}}\ell_{0,T}(\pi)
\eque
denote the geodesic $\calL$ from $(0,0)$ and $(0,T)$. It was proved that in \cite{liu22when} when $\mathcal{L}(0,0;0,T)=L$ goes to infinity, the geodesic $\pi^*$ becomes very rigid and has fluctuations of order $L^{-1/4}$, and the directed landscape along the geodesic $\mathcal{L}(0,0;\pi^*(t),t)$ fluctuates of order $L^{1/4}$. 
Here we have the same issue regarding the definition of conditioning on the event $\calL(0,0;0,T) = L$, and again it is understood as $\proba(\cdot\mid \calL(0,0;0,T) = L) = \lim_{\epsilon\downarrow0}\proba(\cdot\mid\calL(0,0;0,T)\in(L-\epsilon,L+\epsilon))$. See \cite{liu22when}. 
(In \cite{liu22when}, an exact formula for the density function of $\pi^*(t)$ for a fixed $t$ without conditioning is known, and hence as $\epsilon\downarrow0$ it has a well-defined limit by examining the density formula. For our conjectures later, we need to assume implicitly $\pi^*(t)$ exists jointly for different $t_1,\dots,t_d$.)

More precisely, the marginal conditional convergence was established as follows.
\begin{Thm}[\cite{liu22when}]
	\label{thm:3}	
For any  $\tau\in(0,1),T>0$, as $L\to\infty$, 
	\begin{multline*}
		\law\pp{\pp{\frac{2L^{1/4}\pi^*(\tau T)}{T^{3/4}}, \quad \frac{\mathcal{L}(0,0;\pi^*(\tau T),\tau T)-\tau L}{T^{1/4}L^{1/4}}} \mmid \mathcal{L}(0,0;0,T)=L} \\
		\to 
	\law\pp{\pp{\sqrt{\tau(1-\tau)}Z_1,\sqrt{\tau(1-\tau)}Z_2}},
	\end{multline*}
where $Z_1$ and $Z_2$ are two independent standard Gaussian random variables.
\end{Thm}

Now, we relate the KPZ fixed point $\H$ to the directed landscape $\calL$. 
 Then the KPZ fixed point $\H(x,t)$ can be expressed as (see \cite[Corollary 4.2]{nica20one-sided})
\[
\H(x,t)=\sup_{y\in\R} \{\h_0(y) + \mathcal{L}(y,0;x,t)\}, \quad x\in\R, t\ge 0,
\]
where $\h_0(y)$ is the initial condition, and the above is understood as equal in distribution for two processes indexed by $t$ taking values in the space of upper-semicontinuous functions. 
In particular, with  the step initial condition $\h_0(y)=-\infty\cdot \inddd{y\ne 0}$, we have the following representation of $\H_\step$ in terms of $\calL$
\begin{equation}
\label{eq:relation_Hstep}
\ccbb{\H_\step(x,t)}_{x\in\R,t>0}\eqd \ccbb{\mathcal{L}(0,0;x,t)}_{x\in\R,t>0}.
\end{equation}
Note that we restrict to $t>0$ so the above is understood as equal in distribution for random elements in the space $C(\R\times(0,\infty))$. Similarly, $\eqd$ in the sequel stands for equal in distribution with respect to the corresponding space of continuous functions. Moreover, the conditional law of $\{\H_\step(x,t)\}_{x\in\R,t>0}$ given $\H_\step(0,T) = L$ is the same as the conditional law of $\{\calL(0,0;x,t)\}_{x\in\R,t>0}$ given $\calL(0,0;0,T) = L$.

Now, thanks to \eqref{eq:relation_Hstep},  
\[
\ccbb{\max_{\rx\in\R}\frac{\mathcal{L}(0,0;\frac{\rx T^{3/4}}{{2}L^{1/4}},\tau T)-\tau L}{T^{1/4}L^{1/4}}}_{\tau\in(0,1)} \eqd \ccbb{\max_{\rx\in\R}\frac{\H_\step(\frac{\rx T^{3/4}}{{2}L^{1/4}},\tau T)-\tau L}{T^{1/4}L^{1/4}}}_{\tau\in(0,1)}, 
\]
Conjecture \ref{conj:03} says that the left-hand side above converges in distribution to a Brownian bridge, with respect to the conditional law given $\calL(0,0;0,T) = L$, as $L\to\infty$. 

As we argued in Remark \ref{rem:unique}, 
the argmax process of $\H_\step$ with respect to the conditional law
is expected to exist uniquely. Therefore, so is the corresponding argmax process for $\calL$. 
Set $\rx^*(\tau T) = \argmax_{\rx \in\R} \calL(0,0;\frac{\rx T^{3/4}}{2L^{1/4}},\tau T)$. So we now know 
\[
\frac{\calL(0,0;\rx^*(\tau T)\frac{T^{3/4}}{2L^{1/4}},\tau T)-\tau L}{T^{1/4}L^{1/4}}\qmand \frac{\calL(0,0;\pi^*(\tau T),\tau T) - \tau L}{T^{1/4}L^{1/4}} 
\]
 have the same conditional limit distribution (for the convergence of the second, recall Theorem \ref{thm:3}). In view of the uniqueness of the argmax process, it is plausible to expect the following.

\begin{Conj}
	\label{conj:01} For every $\tau\in(0,1), T>0$, conditionally given $\mathcal{L}(0,0;0,T)= L$, as $L\to\infty$,
\[
		\frac{2L^{1/4}\pi^*(\tau T)}{T^{3/4}} -\argmax_{\rx\in\R}  \calL\pp{0,0;\frac{\rx T^{3/4}}{2L^{1/4}},\tau T}\to 0,
\]
	in probability. 
\end{Conj}

The above conjecture would elaborate the rigidity phenomenon from a different aspect: the definition of the geodesic $\pi^*$ via \eqref{eq:pi*} is much more involved than the argmax process, and yet they become the same in the limit.

Now, following Conjectures~\ref{conj:03} and~\ref{conj:01} we arrive at the following conjecture. This would extend the marginal convergence (for fixed $\tau\in(0,1)$) in  Theorem \ref{thm:3}.

\begin{Conj}
	\label{conj:04}
	 	For all $T>0$, as $L\to\infty$, 
	\begin{multline*}
		\law\pp{\ccbb{\frac{2L^{1/4}\pi^*(\tau T)}{T^{3/4}}, \quad \frac{\calL(0,0;\pi^*(\tau T),\tau T)-\tau L}{T^{1/4}L^{1/4}}}_{\tau\in(0,1)} \mmid \mathcal{L}(0,0;0,T)=L} \\
		 \to 
	\law\pp{\ccbb{\B_1^\rmbr(\tau),\B_2^\rmbr(\tau)}_{\tau\in(0,1)}},
	\end{multline*}
where $\B_1^\rmbr$ and $\B_2^\rmbr$ are two 
i.i.d.~Brownian bridges
over interval $[0,1]$.

	\end{Conj}
\begin{Rem}
In \cite[right before Section 9]{basu19connecting}, the authors suggested that for the closely related
exponential last passage percolation on $\mathbb Z^2$, 
Brownian bridge might arise as the scaling limit of the transversal fluctuations of the geodesic in the large deviation regime. 
\end{Rem}

{\em The paper is organized as follows.} In Section \ref{sec:prelim}, we recall the formula we shall work with from \cite{liu22multipoint}. In Sections \ref{sec:step} and \ref{sec:flat} we prove Theorem \ref{thm:2} for the step and flat initial conditions, respectively. 

\subsection*{Acknowledgements} 
The authors would like to thank Jinho Baik, Duncan Dauvergne, Shirshendu Ganguly, Yier Lin, Jeremy Quastel, and Daniel Remenik for very helpful comments and suggestions.
ZL's research was partially supported by
the University of Kansas Start Up Grant, the University
of Kansas New Faculty General Research Fund, Simons Collaboration Grant No.~637861, and NSF grant
DMS-1953687. 
 YW's research was partially supported by Army Research Office (W911NF-20-1-0139), USA.

\section{Preliminaries}\label{sec:prelim}
Recall that we let $\H$ denote the KPZ fixed point with $\mathsf h_0$ its initial condition, and it satisfies  the famous $1:2:3$-scaling-invariance property. Namely, 
\equh\label{eq:123 scaling}
\ccbb{\lambda \H\pp{\lambda^{-2}x,\lambda^{-3}t; \lambda^{-1}\mathsf h_0(\lambda^2\cdot)}}_{x\in\R, t>0} \eqd \ccbb{\H\pp{x,t;\mathsf h_0(\cdot)}}_{x\in\R,t>0}.
\eque
The invariance property holds for general $\mathsf h_0$. We shall work with $\H_\step$ (with $\mathsf h_0(x) = -\infty\inddd{x\ne 0}$) and $\H_\flat$ (with $\mathsf h_0(x) = 0$) specifically in this paper.

\subsection{Explicit formulas for marginal distributions}
In the spatial direction,  for fixed $t>0$, \cite{prahofer02scale} proved that $\{\H_\step(x,t)\}_{x\in\R}$ is the so-called Airy$_2$ process (minus a parabola).
The process $\{\H_\step(x,1)\}_{x\in\R}+x^2$  is a stationary process, and  it  is well-known that the marginal law, say, $\H_\step(0,1)$, has the Tracy--Widom GUE distribution. Namely,
let the function $u= u(x)$ be the Hastings--McLeod solution to the Painlev\'e-II equation $u'' = 2u^3 + x u$ that satisfies the boundary condition $u(x) \sim {\rm Ai}(x)$ as $x\to\infty$, where ${\rm Ai}(x)$ is the Airy function satisfying ${\rm Ai}''(x) = x{\rm Ai}(x)$. Then, it is well-known that
\[
u(x)\sim {\rm Ai}(x) \sim \frac1{2\sqrt \pi x^{1/4}}e^{-\frac23 x^{3/2}}, \mmas x\to\infty.
\]
Throughout, we write $a(x)\sim b(x)$ as $x\to\infty$ if $\lim_{x\to\infty}a(x)/b(x) = 1$. Then, 
\[
\proba\pp{\H_\step(0,1)\le L} = F_\GUE(L)= \exp\pp{-\int_L^\infty (\ell-L)u^2(\ell)\d\ell}.
\]
See \cite{baik08asymptotics} for more details and more related asymptotics. 
In particular, 
 the corresponding probability density function $p_{\step,0,1}(L)$ satisfies
\equh\label{eq:p_GUE}
p_{\step,0,1}(L) = p_\GUE(L) \sim \frac1{8\pi L}e^{-\frac43L^{3/2}} \mmas L\to\infty.
\eque

We also need later the GOE Tracy--Widom distribution, which has cumulative distribution function
\[
F_{\GOE}(L) = F_{\GUE}^{1/2}(L)\exp\pp{-\frac12\int_L^\infty u(s)ds},
\]
with the function $u$ as before. Therefore, its probability density function satisfies
 \[
 p_{\GOE}(L)\sim \frac1{4\sqrt \pi L^{1/4}}e^{-\frac 23 L^{3/2}} \mmas L\to\infty.
 \]
 This distribution is related to $\H_\flat$ via
$\proba(\H_\flat(0,1)\le L) = F_\GOE(2^{2/3}L)$, and hence
\[ p_{\flat,0,1}(L) = 2^{2/3}p_\GOE(2^{2/3}L) \sim \frac1{(8\pi\sqrt L)^{1/2}}e^{-\frac43 L^{3/2}} \mmas L\to\infty.
 \]
  
\subsection{Explicit formulas for joint distributions}
For the rest of this section we focus on $\H_\step$. The corresponding derivation for $\H_\flat$ is similar and provided later in Section \ref{sec:flat} when needed. 
Throughout we write $\vv h = (h_1,\dots,h_m)\in\R^m, \vvx = (x_1,\dots,x_m)\in \R^m$, and
\[
 \vvtau = (\tau_1,\dots,\tau_m) \qmwith 0 = \tau_0<\tau_1<\cdots<\tau_m.
 \]
Since we are proving for convergence of conditional distributions, we shall work with the conditional tail probability
\begin{equation}
\label{eq:Q_step}
\proba\pp{\H_\step(x_j,\tau_j)>h_j, j=1,\dots,m-1\mid \H_\step(x_m,\tau_m) = h_m},
\end{equation}
of which we derive an expression for the rest of this section. 
The formula we shall work with is summarize at the end in Lemma \ref{lem:Q hat}. 
The derivation below is already quite involved. The reason of not working directly with possibly equal time points is that the corresponding formula would become even more sophisticated; see \citep[Section 2.2.3]{liu22multipoint}.

We start with following formula \eqref{eq:F} from \cite{liu22multipoint}. The formula was not provided explicitly therein, but can be derived from  \cite[Proposition 2.9, Definition 2.25]{liu22multipoint}.
We have

\begin{multline}\label{eq:F}
    \proba\pp{\H(x_1,\tau_1)> h_1,\dots,\H(x_{m-1},\tau_{m-1})>h_{m-1},\H(x_m,\tau_m)\le h_m}\\
 = 
 (-1)^{m-1} \oint_{>1} \cdots \oint_{>1} {\mathsf D}_{\step,\vvx,\vvtau}(\vvz,\vvh)\frac{\d z_1}{2\pi\i z_1(1-z_1)}\cdots\frac{\d z_{m-1}}{2\pi \i z_{m-1}(1-z_{m-1})},
\end{multline}
where each $\oint_{>1}$ is integrated over a circle around the origin in the counter-clockwise direction with radii strictly larger than 1,
\[
\mathsf D_{\step,\vvx,\vvtau}(\vvz,\vvh)  := \sum_{\vvn\in\N_0^m}\frac1{(\vvn!)^2}\mathsf D\topp \vvn_{\step,\vvx,\vvtau}(\vvz,\vvh),
\]
with $\N_0:=\{0,1,\dots,\}$ and
\begin{align*}
\mathsf D\topp\vvn_{\step,\vvx,\vvtau}(\vvz,\vvh) 
&   := (-1)^{n_1+\cdots+n_m} \\
& \quad \times \prodd j2m\left[\prod_{i_j=1}^{n_j}\pp{\frac1{1-z_{j-1}}\int_{C_{j,\rmL}^{\rm in}}\frac{\d\xi_{i_j}\topp j}{2\pi\i}- \frac{z_{j-1}}{1-z_{j-1}}\int_{C_{j,\rmL}^{\rm out}}\frac{\d\xi_{i_j}\topp j}{2\pi\i}}\prod_{i_1=1}^{n_1}\int_{C_{1,\rmL}}\frac{\d\xi_{i_1}\topp 1}{2\pi\i}\right.\nonumber\\
& \quad\times \left.\prod_{i_j=1}^{n_j}\pp{\frac1{1-z_{j-1}}\int_{C_{j,\rmR}^{\rm in}}\frac{\d\eta_{i_j}\topp j}{2\pi\i}- \frac{z_{j-1}}{1-z_{j-1}}\int_{C_{j,\rmR}^{\rm out}}\frac{\d\eta_{i_j}\topp j}{2\pi\i}}\prod_{i_1=1}^{n_1}\int_{C_{1,\rmR}}\frac{\d\eta_{i_1}\topp 1}{2\pi\i}\right]\nonumber\\
& \quad \times \prodd j1{m-1}\pp{\frac{\Delta(\vv\xi\topp j;\vv\eta\topp{j+1})\Delta(\vv\eta\topp j;\vv\xi\topp {j+1})}{\Delta(\vv\xi\topp j;\vv\xi\topp{j+1})\Delta(\vv\eta\topp j;\vv\eta\topp {j+1})} (1-z_j)^{n_j}\pp{1-\frac1{z_j}}^{n_{j+1}}}\nonumber\\
& \quad\times \prodd j1m \pp{\frac{\Delta(\vv\xi\topp j)\Delta(\vv\eta\topp j)}{\Delta(\vv\xi\topp j;\vv\eta\topp j)}}^2
\prodd j1m \prodd {i_j}1{n_j}\frac{f_{\wt x_j,\wt \tau_j}(\xi_{i_j}^{(j)},\wt h_j)}{f_{\wt x_j,\wt \tau_j}(\eta_{i_j}^{(j)},\wt h_j)},
\end{align*}
and the notations are further explained as follows.

First, the second and third lines above are understood as a compact form of linear combinations of $2(n_1+\cdots+n_m)$-multiple contour integrals with respect to $\xi\topp j_{i_j}, \eta\topp j_{i_j}, j=1,\dots,m, i_j=1,\dots,n_j$, and the fourth and fifth row are the integrands. Here and below, we fix $\vvn = (n_1,\dots,n_m)\in\N_0^m$, and for each $j$, $\vv\xi\topp j = (\xi\topp j_1,\dots,\xi\topp j_{n_j}),\vv\eta\topp j = (\eta\topp j_1,\dots,\eta\topp j_{n_j})\in \C^{n_j}$. Throughout we write, with $\vv w = (w_1,\dots,w_k)$ and $\vv w' = (w_1',\dots,w_{k'}')$, 
\[
\Delta(\vv w)   :=\prod_{1\le i<j\le k}(w_j-w_i) \qmand
\Delta(\vv w;\vv w')  := \prodd i1k \prodd {i'}1{k'}(w_i-w_{i'}'),
\]
with the convention $\prod_{i=1}^0(\cdots)  = 1$.
Moreover,  
\begin{align}
	\label{eq:def_f}
	f_{\wt x,\wt \tau}(\zeta,\wt h) : = \exp  \left(-\frac13 \wt \tau\zeta^3 +\wt x \zeta^2 +\wt h\zeta\right).
\end{align}
We also set 
\[
\wt \tau_j := \tau_j-\tau_{j-1},\quad \wt x_j := x_j-x_{j-1}, \quad \wt h_j := h_j-h_{j-1}, \quad j=1,\dots,m,
\]
and $x_0=h_0=\tau_0=0$.
In particular, $\summ j1m \wt\tau_j = \tau_m$ and $\summ j1m \wt h_j = h_m$. 

The choice of the contours $C_{j,\rmL/\rmR}^{\rmin/\rmout}$ is delicate and plays a crucial role in our analysis later, and we summarize some key features in the following remark.
\begin{Rem}
Throughout we only need and work with the formula with $\wt\tau_j>0$ for all $j=1,\dots,m$. 
We describe the choice of the contours (actually, only the direction towards infinity matters regarding integrability of the multiple integrals), starting with 
\[
C_{m,\rmL}^{\rm in},\dots,C_{2,\rmL}^{\rm in}, C_{1,\rmL}, C_{2,\rmL}^{\rm out},\dots, C_{m,\rmL}^{\rm out}.
\]
These contours in the left half-plane are non-intersecting. To guarantee integrability, we may take each starting from $e^{-2\pi \i/3}\infty$ to $e^{2\pi\i/3}\infty$, located from left to the right. The choice of the angle is such that the leading term $-\wt\tau_j(\xi_{i_j}\topp j)^3$ in the exponential function has a negative real part as $\xi_{i_j}\topp j$ tends to infinity along the contour (in both directions).   Similarly, one can pick, $C_{m,\rmR}^{\rm out}, \dots, C_{2,\rmR}^{\rm out}, C_{1,\rmR},C_{2,\rmR}^{\rm in},\dots,C_{m,\rmR}^{\rm in}$ are contours in the right half-plane, non-intersecting, each starting from $e^{-\pi \i/3}\infty$ to $e^{\pi\i/3}\infty$, and located from left to the right. 
\end{Rem}
From the above we derive an expression of the tail probability that we work with. The formula takes a similar form as in \eqref{eq:F}.  
Introduce
\begin{multline}\label{eq:Pi_n}
\Pi_\vvn(\vec{\vv\xi}\topp m, \vec{\vv\eta}\topp m)\equiv \Pi_\vvn(\vv\xi\topp 1,\cdots,\vv\xi\topp m; \vv\eta\topp1,\cdots,\vv\eta\topp m)\\
:=(-1)^{n_1+\cdots+n_m} \prodd j1{m-1}\frac{\Delta(\vv\xi\topp j;\vv\eta\topp{j+1})\Delta(\vv\eta\topp j;\vv\xi\topp {j+1})}{\Delta(\vv\xi\topp j;\vv\xi\topp{j+1})\Delta(\vv\eta\topp j;\vv\eta\topp {j+1})} \prodd j1m \pp{\frac{\Delta(\vv\xi\topp j)\Delta(\vv\eta\topp j)}{\Delta(\vv\xi\topp j;\vv\eta\topp j)}}^2 
\cdot  \sum_{i_m=1}^{n_m}\pp{\xi\topp m_{i_m}-\eta\topp m_{i_m}},
\end{multline}
and 
\begin{equation}
\label{eq:whatD}
\begin{split}
 \what{\mathsf D}\topp\vvn_{\step,\vvx,\vvtau}(\vvz,\vvh) 
  & :=  \frac{\partial}{\partial h_m}\mathsf D_{\step,\vvx,\vvh}\topp{\vvn}(\vvz,\vvh) \\
& = 
  \prodd j2m \Bigg[ (1-z_{j-1})^{n_{j-1}}\pp{1-\frac1{z_{j-1}}}^{n_{j}} \\
  &\quad \times  \prod_{i_j=1}^{n_j}\pp{\frac1{1-z_{j-1}}\int_{C_{j,\rmL}^{\rm in}}\frac{\d\xi_{i_j}\topp j}{2\pi\i}- \frac{z_{j-1}}{1-z_{j-1}}\int_{C_{j,\rmL}^{\rm out}}\frac{\d\xi_{i_j}\topp j}{2\pi\i}}\times\prod_{i_1=1}^{n_1}\int_{C_{1,\rmL}}\frac{\d\xi_{i_1}\topp 1}{2\pi\i}\\
  &\quad \times \prod_{i_j=1}^{n_j}\pp{\frac1{1-z_{j-1}}\int_{C_{j,\rmR}^{\rm in}}\frac{\d\eta_{i_j}\topp j}{2\pi\i}- \frac{z_{j-1}}{1-z_{j-1}}\int_{C_{j,\rmR}^{\rm out}}\frac{\d\eta_{i_j}\topp j}{2\pi\i}}\times \prod_{i_1=1}^{n_1}\int_{C_{1,\rmR}}\frac{\d\eta_{i_1}\topp 1}{2\pi\i}\Bigg]\\
  &\quad \times \Pi_\vvn(\vec{\vv\xi}\topp m, \vec{\vv\eta}\topp m)\times  \prodd j1m \prodd {i_j}1{n_j}\frac{f_{\wt x_j,\wt \tau_j}(\xi_{i_j}^{(j)},\wt h_j)}{f_{\wt x_j,\wt \tau_j}(\eta_{i_j}^{(j)},\wt h_j)}.
\end{split}
\end{equation}
\begin{Lem}
\label{lem:Q hat}
We have
\equh\label{eq:Q conditional}
\proba\pp{\H_\step(x_j,\tau_j)>h_j, j=1,\dots,m-1\mid \H_\step(x_m,\tau_m) = h_m}
= \frac{\what Q_{\step,\vvx,\vvtau}(\vvh)}{p_{\step,0,T}(h_m)},\quad \mfa \vvh\in\R^m,
\eque
where $p_{\step,0,T}$ is the probability density function of $\H_\step(0,T)$, 
\[
\what Q_{\step,\vvx,\vvtau}(\vvh) := \sum_{\vvn\in\N_0^m}\frac1{(\vvn!)^2} \what Q_{\step,\vvx,\vvtau}\topp \vvn(\vvh), 
\]with
\equh\label{eq:Q_step n}
\what Q_{\step,\vvx,\vvtau}\topp\vvn(\vvh) := (-1)^{m-1}\oint_{>1} \cdots \oint_{>1} {\what{\mathsf D}\topp\vvn}_{\step,\vvx,\vvtau}(\vvz,\vvh)\frac{\d z_1}{2\pi\i z_1(1-z_1)}\cdots\frac{\d z_{m-1}}{2\pi \i z_{m-1}(1-z_{m-1})}, 
\eque
for each $\vvn\in\N^m$, where
where each $\oint_{>1}$ is integrated over a contour around the origin in the counter-clockwise direction.
\end{Lem}
\begin{proof}
Note that the function $\what Q_{\step,\vvx,\vvtau}$ can be defined as 
\[ \what Q_{\step,\vvx,\vvtau}(\vvh) := \frac{\partial}{\partial h_m}\proba\pp{\H(x_1,\tau_1)> h_1,\dots,\H(x_{m-1},\tau_{m-1})>h_{m-1},\H(x_m,\tau_m)\le h_m}.
\]
One then readily checks the stated formula.
Note that 
the only difference between $\what{\mathsf D}_{\step,\vvx,\vvtau}\topp\vvn$ and $\mathsf D_{\step,\vvx,\vvtau}\topp\vvn$ is the extra factor in the second line of \eqref{eq:Pi_n} (note that we are differentiating with respect to $h_m$, not $\wt h_m = h_m-h_{m-1}$). 
\end{proof}

\section{Proof for the case with step initial condition}\label{sec:step}

Thanks to the scaling invariance property of the KPZ fixed point in~\eqref{eq:123 scaling}, it is sufficient to prove~\eqref{thm:2} in Theorem \ref{thm:2}  for the case of $T=1$, which we assume throughout this section.
We shall then prove the following restatement of Theorem \ref{thm:2} in the case with step initial condition:
\begin{multline}\label{eq:limit step}
\calL\pp{\ccbb{\frac{\H_\step(\frac{\rx}{\sqrt 2L^{1/4}},\tau)-\tau\H_\step(0,1)}{\sqrt 2L^{1/4}} }_{\rx\in\R,\tau\in(0,1)}\mmid  \H_\step(0,1) = L}\\
 \fddto \calL\pp{\ccbb{\min\ccbb{\B^{\rmbr}_1(\tau)+\rx,\B^{\rmbr}_2(\tau)-\rx}}_{\rx\in\R,\tau\in(0,1)}},
\end{multline}
as $L\to\infty$, where $\B^\rmbr_1$ and $\B^\rmbr_2$ are two independent Brownian bridges.

We shall prove the above in two steps. For most part of this section before Section \ref{sec:ZP}, we consider convergence of finite-dimensional distributions at  {\em distinct} time points $\tau_0,\dots,\tau_m$. Then in Section \ref{sec:ZP}, we prove the general case when some time points may be the same.

Fix $m\ge 2$, and set
\[
\begin{split}
&0=\tau_0<\tau_1<\cdots<\tau_m = 1,\\
& \rh_1,\dots,\rh_{m-1}\in\R, \quad \rh_0=\rh_m = 0, \\
& \rx_1,\dots,\rx_{m-1}\in\R, \quad \rx_0=\rx_m = 0,
\end{split}
\]
and 
\begin{equation}
\label{eq:h_L,j}
h_{L,j}  =\tau_j L +\rh_j\cdot \sqrt{2}L^{1/4}\qmand 
x_{L,j}   =\rx_j \cdot\frac{1}{\sqrt{2}L^{1/4}},\quad
j=1,\cdots,m, \quad L>0.
\end{equation}
We also write $\vvrx=(\rx_1,\cdots,\rx_{m-1})$, ${\vvrh}=(\rh_1,\cdots,\rh_{m-1})$, $\vvx_L=(x_{L,1},\cdots,x_{L,m})$, $\vvh_L=(h_{L,1},\cdots,h_{L,m})$ and $\vvtau=(\tau_{1},\cdots,\tau_{m})$.
Then, in accordance to \eqref{eq:Q_step} we consider,
\[
\proba\pp{\H_\step(x_{L,j},\tau_j)>h_{L,j}, j=1,\dots,m-1\mmid \H_\step(0,1) = L}
 = \frac{\what Q_{\step,\vvx_L,\vvtau}(\vvh_L)}{p_{\GUE}(L)},
\]
with $\what Q_{\step,\vvx_L,\vvtau}(\vvh_L)$ in Lemma \ref{lem:Q hat}. 
Recall the asymptotic probability density function of $p_{\rm GUE}$ in \eqref{eq:p_GUE}.
 Write the corresponding tail probability in the limit as
\[
\mathsf Q_{\vvrx,\vvtau}(\vvrh)  := \proba\pp{\min\ccbb{\B_1^{\rmbr}(\tau_j)+\rx_j,  \B_2^{\rmbr}(\tau_j)-\rx_j}>\rh_j,j=1,\dots,m-1}.
\]
Therefore, \eqref{eq:limit step} is equivalent to
 \equh\label{eq:limit pdf}
\frac{\what Q_{\step,\vvx_L,\vvtau}(\vvh_L)}{p_\GUE(L)} \sim 
\frac{\what Q_{\step,\vvx_L,\vvtau}(\vvh_L)}{(8\pi L)\inv \exp\pp{-\frac43L^{3/2}}} \to \mathsf Q_{\vvrx,\vvtau}(\vvrh) \mmas L\to\infty.
\eque

The goal of the rest of this section is to prove \eqref{eq:limit pdf}. We start by recalling $
\what Q_{\step,\vvx_L,\vvtau}(\vvh_L) = \sum_{\vvn\in\N_0^m}\frac1{(\vvn!)^2} \what Q_{\step,\vvx_L,\vvtau}\topp \vvn(\vvh_L)$ in Lemma \ref{lem:Q hat}.
Then,  \eqref{eq:limit pdf} follows immediately from  the following three lemmas, with $\tau_j$ all distinct.  
 
\begin{Lem}\label{lem:1}
If $\vvn\in\N_0^m\setminus\N^m$, that is, $n_j = 0$  for some $j=1,\dots,m$, then $\what Q_{\step,\vvx,\vvtau}\topp\vvn(\vvh)= 0$.
\end{Lem}
\begin{Lem}\label{lem:2}
As $L\to\infty$, with $\vv1 \equiv (1,\dots,1)\in\N^m$,
\[
\what Q\topp{\vv1}_{\step,\vvx_L,\vvtau}(\vvh_L) \sim \frac1{8\pi L}e^{-\frac 43L^{3/2}}\times \mathsf Q_{\vvrx,\vvtau}(\vvrh).
\]
\end{Lem}
\begin{Lem}\label{lem:3}
For every $\epsilon'\in (0, \min_{j=1,\dots,m}\wt\tau_j)$, there exists a constant $C>0$ such that
\[
\sum_{\vvn\in\N^m\setminus\{\vv1\}}\frac1{(\vvn!)^2}\abs{\what Q_{\step,\vvx_L,\vvtau}\topp\vvn(\vvh_L)} \le C  \exp\pp{-\frac{4(1+\epsilon')}3L^{3/2}},
\]
for all $L$ large enough.
\end{Lem}
We prove the three lemmas in Sections \ref{sec:3.1}, \ref{sec:3.2} and \ref{sec:3.3} respectively. For the case $\tau_j$ are not all distinct, the theorem then follows from a bootstrap argument by combining the above and a general lemma in Section \ref{sec:ZP}, which is of its own interest.

Before, in Section \ref{sec:Bb} we present some auxiliary results on Brownian bridges. 
 Throughout, we let $C>0$ denote a constant that may change from line to line, but not depending on $\vvn$ nor $L$. 
\subsection{Auxiliary formula of Brownian bridge}\label{sec:Bb}
Let $\phi_\sigma(x) = (\sqrt{2\pi}\sigma)\inv e^{-x^2/(2\sigma^2)}$ denote the probability density function of a centered Gaussian random variable with variance $\sigma^2, \sigma>0$. 
It is well-known that the joint probability density function of a Brownian bridge $\B^{\rmbr}$ at times $0=a_0<a_1<\cdots<a_{m-1}<a_m = 1$ has the formula
\[
\mathsf p^{\rmbr}_{a_1,\dots,a_{m-1}}\pp{b_1,\dots,b_{m-1}} = \sqrt{2\pi} \prodd j1m \phi_{a_j-a_{j-1}}(b_j-b_{j-1}),
\]
for $b_1,\dots,b_{m-1}\in\R, b_0 =  b_m = 0$.

We will need a formula for the joint cumulative distribution function of $\B^{\rmbr}$ at different times, as shown in the following lemma. Define
\begin{equation}
	\label{eq:def_rf}
\rf(u;a,b):=\exp\left(\frac{1}{2}au^2+bu\right),\qquad u\in\C, a, b\in\R.
\end{equation}
Note the simple identity
\begin{equation}
	\label{eq:aux03}
	\phi_a(b)=\int_{\Gamma} \rf(u;a,b)\frac{\d u}{2\pi\i},
\end{equation}
if $a>0$ and $\Gamma$ is an arbitrary contour parallel to the $y$-axis with upward orientation.

\begin{Lem}
\label{lem:bridge}
Let
$\Gamma_1,\dots,\Gamma_m$ be disjoint contours listed from left to right, each parallel to the $y$-axis with upwards orientation. We have for all $0=a_0<a_1<\cdots<a_m =  1$ and $b_1,\cdots,b_{m-1}\in\R$, $b_0=b_m=0$,  
\begin{equation}
\label{eq:BR_cdf}
\sqrt{2\pi} \int_{\vec\Gamma} \frac{\prod_{j=1}^m \rf(u_j;a_j-a_{j-1},b_{j}-b_{j-1})}{\prod_{j=1}^{m-1}(u_{j+1}-u_j)}
\frac{\d \vvu}{(2\pi\i)^m}
= \proba\pp{\B^{\rmbr}({a_{j}})>b_j, j=1,\dots,m-1},
\end{equation}
with abbreviation $\vec\Gamma \equiv \Gamma_1\times\cdots\times\Gamma_m$ and $\d\vvu \equiv \d u_1\cdots\d u_m$ (with the convention $u_1\in\Gamma_1,\dots,u_m\in\Gamma_m$). 
\end{Lem}
\begin{proof}
We view the left-hand side of~\eqref{eq:BR_cdf} as a function of $b_1,\cdots,b_{m-1}$, denoted by $Q(b_1,\cdots,b_{m-1})$ below.
First we note that $Q(b_1,\cdots,b_{m-1})$ is well-defined since the integrand is uniformly bounded and decays super-exponentially fast along the integration contours. The function $Q$ is continuous on each $b_i$. Moreover, if we write the integrand of left-hand side of \eqref{eq:BR_cdf} as
\begin{equation}
\label{eq:aux01}
\frac{\prod_{j=1}^m \rf(u_j;a_j-a_{j-1},0)}{\prod_{j=1}^{m-1}(u_{j+1}-u_j)}\cdot \prod_{j=1}^{m-1}\exp\left(b_j(-u_{j+1}+u_j)\right),
\end{equation}
we can see that the integrand converges to $0$ uniformly if one of $b_j\to+\infty$ since $\rmRe(-u_{j+1}+u_j)<0$ by our choice of the contours. Hence $Q(b_1,\cdots,b_{m-1})\to 0$ and
\begin{equation}
	\label{eq:aux04}
	Q(b_1,\cdots,b_{m-1})-\proba\pp{\B^{\rmbr}({a_{j}})>b_j, j=1,\dots,m-1}\to 0
\end{equation}
as any of $b_j\to+\infty$.

Now we consider 
$(\partial^{m-1}/{\partial b_1\cdots \partial b_{m-1}})Q$.
We use the form~\eqref{eq:aux01} of the integrand of $Q$, and
 change the order of derivative and integral to obtain
\begin{equation}
\label{eq:aux02}
\frac{\partial^{m-1}Q(b_1,\cdots,b_{m-1})}{\partial b_1\cdots \partial b_{m-1}}=\sqrt{2\pi}(-1)^{m-1} \int_{\vec\Gamma} \prod_{j=1}^m \rf(u_j;a_j-a_{j-1},b_{j}-b_{j-1})\frac{\d\vvu}{(2\pi\i)^m}.
\end{equation}
To justify the change of the order of integration and differentiation, it suffices to notice that the product $\prod_{j=1}^m \rf(u_j;a_j-a_{j-1},b_{j}-b_{j-1})$ is continuous in $b_j$'s and in $u_j$'s, and decays super-exponentially fast along the integration contours. 

Now we combine~\eqref{eq:aux02} and~\eqref{eq:aux03}, and obtain
\begin{align*}
	\frac{\partial^{m-1}Q(b_1,\cdots,b_{m-1})}{\partial b_1\cdots \partial b_{m-1}}&=\sqrt{2\pi}(-1)^{m-1}\prodd j1m \phi_{a_j-a_{j-1}}(b_j-b_{j-1})\\
	&=(-1)^{m-1}\mathsf p^{\rmbr}_{a_1,\cdots,a_{m-1}}\pp{b_1,\dots,b_{m-1}}\\
	&=\frac{\partial^{m-1}\proba\pp{\B^{\rmbr}({a_{j}})>b_j, j=1,\dots,m-1}}{\partial b_1\cdots \partial b_{m-1}}.
\end{align*}
This and~\eqref{eq:aux04} then  imply $Q(b_1,\cdots,b_{m-1})=\proba\pp{\B^{\rmbr}({a_{j}})>b_j, j=1,\dots,m-1}$. This completes the proof.
\end{proof}

\subsection{Proof of Lemma \ref{lem:1}}\label{sec:3.1}
The case $n_m = 0$ is trivial since the last factor $\sum_{i_j=1}^{n_m} (\xi_{i_m}\topp m-\eta_{i_m}\topp m)$ in  \eqref{eq:Pi_n} becomes zero. 
Note that if $n_m\ne 0$ and $\vvn\in\N_0^m\setminus\N^m$,  then necessarily for some $j=2,\dots,m$, $(n_{j-1},n_j) = (0,n_j)$ with $n_j\in\N$. Consider such a $j$. We have
\begin{align*}
\oint_{>1} & \what {\mathsf D}_{\step,\vvx,\vvtau}\topp{\vvn}(\vvz,\vvh) \frac{\d z_{j-1}}{2\pi \i\cdot z_{j-1}(1-z_{j-1})}
 = 
\oint_{>1} \frac{\d z_{j-1}}{2\pi \i\cdot z_{j-1}(1-z_{j-1})}\\
&\times \Bigg\{\pp{1-\frac1{z_{j-1}}}^{n_j}
\prod_{i_j=1}^{n_j} \pp{\frac1{1-z_{j-1}}\int_{C_{j,\rmL}^\rmin}\frac{\d \xi_{i_j}\topp j}{2\pi\i} - \frac{z_{j-1}}{1-z_{j-1}}\int_{C_{j,\rmL}^\rmout}\frac{\d\xi_{i_j}\topp j}{2\pi \i}}\\
&\times \pp{\frac1{1-z_{j-1}}\int_{C_{j,\rmR}^\rmin}\frac{\d \eta_{i_j}\topp j}{2\pi\i} - \frac{z_{j-1}}{1-z_{j-1}}\int_{C_{j,\rmR}^\rmout}\frac{\d\eta_{i_j}\topp j}{2\pi \i}}\Bigg\}\times (\cdots),
\end{align*}
where we skipped the factor without the variable $z_{j-1}$. Note that the integrand is a rational function of $z_{j-1}$ with degree at most $-2-n_j\le -3$. We also note that the contour of the integral with respect to $z_{j-1}$ can be deformed to an arbitrary large circle without crossing any pole. Thus the above integral equals to zero by enlarging the contour to infinity.
\subsection{Proof of Lemma \ref{lem:2}}\label{sec:3.2}
We first simplify $\what Q\topp{\vv1}_{\step,\vvx,\vvtau}$ (recalling \eqref{eq:Q_step n}) in Lemma \ref{lem:Q_hat_11} below. 
Now with $\vvn = \vv1$, for each $j=1,\dots,m$, $\vv\xi\topp j = (\xi\topp j_1)$ is one-dimensional. So we simply write $\xi\topp j \equiv \xi\topp j_1, j=1,\dots,m$ and similarly for $\eta\topp j$, and 
\[
\vec\xi\topp m = (\xi\topp1,\dots,\xi\topp m) \qmand \vec\eta\topp m = (\eta\topp1,\dots,\eta\topp m).
\]

\begin{Lem}
\label{lem:Q_hat_11}	
We have
\equh
	\label{eq:Q_hat_11}
	\what Q_{\step,\vvx_L,\vvtau}\topp{\vv1}(\vvh_L) = (-1)^{m-1}\int_{\vec C}  
	\Pi_\vv1(\vec\xi\topp m,\vec\eta\topp m)\times 
	\prodd j1m \frac{f_{\wt x_{L,j},\wt \tau_j}(\xi^{(j)},\wt h_{L,j})}{f_{\wt x_{L,j},\wt \tau_j}(\eta^{(j)},\wt h_{L,j})}
	\frac{\d\vec\xi\topp m\d\vec\eta\topp m}{(2\pi \i)^{2m}},
	\eque
	with
	\[ \vec C := C_{1,\rmL}\times C_{2,\rmL}^\rmout\times\cdots\times C_{m,\rmL}^\rmout\times C_{1,\rmR}\times C_{2,\rmR}^\rmout\times\cdots\times C_{m,\rmR}^\rmout,
\]
and
\begin{equation}
	\label{eq:Pi_1}
	\Pi_{\vv1}(\vec\xi\topp m,\vec\eta\topp m):= 
	(-1)^m \prodd j1{m-1}\frac{(\xi\topp j-\eta\topp{j+1})(\eta\topp j-\xi\topp{j+1})}{(\xi\topp j-\xi\topp{j+1})(\eta\topp j-\eta\topp{j+1})} \frac1{(\xi\topp j-\eta\topp j)^2}
	\times \frac{1}{\xi\topp m-\eta\topp m}.
\end{equation}
\end{Lem}
\begin{proof}
We first remark that~\eqref{eq:Pi_1} is the same as $\Pi_\vvn$ in~\eqref{eq:Pi_n} when $\vvn=\vv1$.
For \eqref{eq:Q_hat_11}, recall~\eqref{eq:Q_step n} and we choose the contours to be $|z_1|=|z_2|=\cdots=|z_m|=R$ and let $R$ become large. Note that if we expand the product of the integrals in~\eqref{eq:whatD} with $\vvn=\vv1$ as a summation of $2m$-multiple integrals, we find that $\what {\mathsf D}\topp{\vv1}_{\step,\vvx,\vvtau}(\vvz,\vvh)$ has the following leading term from the integral corresponding to $\vec C$ of order $O(R^{m-1})$, 
\[
\prodd j1{m-1}(1-z_j)\int_{\vec C}\Pi_\vv1(\vec\xi\topp m,\vec\eta\topp m)\times \prodd j1m \frac{f_{\wt x_j,\wt \tau_j}(\xi^{(j)},\wt h_j)}{f_{\wt x_j,\wt \tau_j}(\eta^{(j)},\wt h_j)}\frac{\d\vec\xi\topp m}{(2\pi\i)^m}\frac{\d\vec\eta\topp m}{(2\pi\i)^m},
\]
	plus an error term (the sum of the rest integrals) of order $O(R^{m-2})$.
	The integration of the leading term above with respect to $z_j, j=1,\dots,m-1$ yields desired right-hand side of \eqref{eq:Q_hat_11}. Similarly, the integration of each of the rest $2m$-multiple integrals  is bounded by
\[
	\oint_{|z_1|=R}\cdots\oint_{|z_m|=R} O(R^{m-2}) \frac{|\d z_1|}{2\pi R |1-z_1|}\cdots \frac{|\d z_{m-1}|}{2\pi R |1-z_{m-1}|}=O(R^{-1})
\]
since $\int_{|z|=R}\frac{\d z}{|1-z|}$ is bounded by a constant for large $R$. Letting $R\to\infty$, we see that the upper bound goes to zero, and hence the integration of the error term must be zero. We have proved~\eqref{eq:Q_hat_11}.
\end{proof}

Since we do not have any contours $C_{j,\rmL}^{\rm in}$ or $C_{j,\rmR}^{\rm in}$ in the formula~\eqref{eq:Q_hat_11}, we drop the superscripts $\rm out$ in the integration contours of~\eqref{eq:Q_hat_11} in this subsection for notation simplification. In other words, $C_{j,\rmL}=C_{j,\rmL}^{\rm out}$ and $C_{j,\rmR}=C_{j,\rmR}^{\rm out}$ for $j=2,\cdots,m$ in this subsection.

\bigskip
In the next step, we compute the asymptotics of $\what Q_{\step,\vvx_L,\vvtau}\topp{\vv1}(\vvh_L)$ using \eqref{eq:Q_hat_11} and prove Lemma~\ref{lem:2}. We need to deform the integration contours in~\eqref{eq:Q_hat_11} as follows (now depending on $L$): for $j=1,\dots,m$,
\[
C_{j,\rmL,L}:= -\sqrt{L}+ 2^{-1/2}L^{-1/4}\Sigma_{j,\rmL},\quad C_{j,\rmR,L}:=\sqrt{L} +2^{-1/2}L^{-1/4}\Sigma_{j, \rmR}
\]
where
\begin{equation}
\label{eq:Sigma}
	\Sigma_{j,\rmL}:=\{j+ re^{\pm 2\pi\i/3 }: r>0\},\quad \text{and } \Sigma_{j, \rmR}:=-\Sigma_{j, \rmL}.
\end{equation}
The orientations of $\Sigma_{\rmL}$ and $\Sigma_{\rmR}$ in the integrals below are from $\infty e^{-2\pi\i/3}$ to $\infty e^{2\pi \i/3}$, and from $\infty e^{-\pi\i/3}$ to $\infty e^{\pi\i/3}$ respectively. With these contours, 
for $\xi_L^{(j)}\in C_{j,\rmL,L}$ and $\eta_L^{(j)}\in C_{j,\rmR,L}$ 
we can parametrize 
\[
	\xi_L^{(j)} = \xi_L\topp j(u_j)= -\sqrt{L} + 2^{-1/2}L^{-1/4}u_j,\quad \eta_L^{(j)}=\eta_L\topp j(v_j) = \sqrt{L} + 2^{-1/2}L^{-1/4} v_j,
\]
for some $u_j\in \Sigma_{j,\rmL}$ and $v_j\in \Sigma_{j,\rmR}$. 
Recall the definitions of $f$ in~\eqref{eq:def_f} and $\rf$ in~\eqref{eq:def_rf}, and note the scaling of the parameters~\eqref{eq:h_L,j}. We have, after a direct computation,
\begin{align}
\label{eq:f}
	f_{\wt x_{L,j},\wt \tau_j}\left(\xi_L^{(j)},\wt h_{L,j}\right)
	&=\exp\left(-\frac13\wt\tau_j(\xi_L^{(j)})^3 +\wt x_{L,j}(\xi_L^{(j)})^2 +\wt h_{L,j}\xi_L^{(j)}\right)\\
	&=\exp\left(-\frac{2}{3}\wt \tau_jL^{3/2}+\left(\frac{\wt \rx_j}{\sqrt{2}}-\sqrt{2}\wt \rh_j\right)L^{3/4}\right)
	\cdot \rf(u_j;\wt \tau_j,\wt\rh_j-\wt\rx_j)\cdot g_L(u_j;\tilde \tau_j,\tilde\rx_j),\nonumber
\end{align}
and similarly,
\[
		f_{\wt x_{L,j},\wt \tau_j}\left(\eta_L^{(j)},\wt h_{L,j}\right)=\exp\left(\frac{2}{3}\wt\tau_jL^{3/2}+\left(\frac{\wt\rx_j}{\sqrt{2}}+\sqrt{2}\wt\rh_j\right)L^{3/4}\right)
		\cdot \frac1{\rf(v_j;\wt\tau_j,-\wt\rh_j-\wt \rx_j)}\cdot  g_L(v_j;\tilde \tau_j,\tilde\rx_j),
\]
where $\wt\tau_j=\tau_j-\tau_{j-1}$, $\wt x_j=x_j-x_{j-1}$, $\wt h_j=h_j-h_{j-1}$, and $\wt\rx_j=\rx_j-\rx_{j-1}$, $\wt\rh_j=\rh_j-\rh_{j-1}$, and 
\[
	g_L(w;\tilde\tau,\tilde\rx):=\exp\pp{\left(-\frac{1}{6\sqrt{2}}\tilde\tau w^3+\frac{\tilde\rx}{2\sqrt{2}}w^2\right)L^{-3/4}}.
\]
We see that $g_L(w;\tilde\tau,\tilde\rx)$ decays super-exponentially fast when $w\in\Sigma_{j,\rmL}\to\infty$ and grows super-exponentially fast when $w\in\Sigma_{j,\rmR}\to\infty$, if $\tilde\tau>0$.

Note that $\sum_{j=1}^m\wt\tau_j=\tau_m=1$, $\sum_{j=1}^m\wt \rx_j=\rx_m=0$ and $\sum_{j=1}^m\wt \rh_j=\rh_m=0$. Hence taking the product over $j=1,\dots,m$ we have
\begin{align}
	\label{eq:prod_f_xi}
	\prod_{j=1}^mf_{\wt x_{L,j},\wt \tau_j}\left(\xi_L^{(j)},\wt h_{L,j}\right) &=e^{-\frac{2}{3}L^{3/2}}\prod_{j=1}^m\rf(u_j;\wt \tau_j,\wt\rh_j-\wt\rx_j)\cdot g_L(u_j;\tilde \tau_j,\tilde\rx_j),\\
		\label{eq:prod_f_eta}
	\prod_{j=1}^mf_{\wt x_{L,j},\wt \tau_j}\left(\eta_L^{(j)},\wt h_{L,j}\right) & =e^{\frac{2}{3}L^{3/2}}\prod_{j=1}^m\frac1{\rf(v_j;\wt \tau_j,-\wt\rh_j-\wt\rx_j)}\cdot g_L(v_j;\tilde \tau_j,\tilde\rx_j).
\end{align}
We also write~\eqref{eq:Pi_1} as
\begin{equation}
	\label{eq:Pi_1_rewrite}
	\Pi_{\vv1}(\vec\xi_L\topp m,\vec\eta_L\topp m)=2^{m-2}L^{m/2-1}J_L(u_1,\cdots,u_m,v_1,\cdots,v_m)
\end{equation}
where
\begin{equation}
	\label{eq:J01}
	\begin{split}
	&J_L(u_1,\cdots,u_m,v_1,\cdots,v_m)\\
	&:=
	\prod_{j=1}^{m-1}\frac{(1-2^{-3/2}(u_j-v_{j+1})L^{-3/4})(1+2^{-3/2}(v_j-u_{j+1})L^{-3/4})}{(u_j-u_{j+1})(v_j-v_{j+1})(1-2^{-3/2}(u_j-v_{j})L^{-3/4})^2}\cdot \frac{1}{1-2^{-3/2}(u_m-v_{m})L^{-3/4}}.
\end{split}
\end{equation}
Note also that $\d\vec\xi\topp m\d\vec\eta\topp m = 2^{-m}L^{-m/2}\d\vv u\d\vvv$. 
By inserting~\eqref{eq:prod_f_xi},~\eqref{eq:prod_f_eta} and~\eqref{eq:Pi_1_rewrite} in~\eqref{eq:Q_hat_11}, we arrive at
\begin{multline*}
	\frac{\what Q_{\step,\vvx_L,\vvtau}\topp{\vv1}(\vvh_L)}{(8\pi L)^{-1}e^{-\frac{4}{3}L^{3/2}}}=2\pi (-1)^{m-1}\int_{\vec\Sigma_{\rmL}\times\vec\Sigma_{\rmR}}
	\prod_{j=1}^m\pp{\rf(u_j;\wt \tau_j,\wt\rh_j-\wt\rx_j)\cdot \rf(v_j;\wt \tau_j,-\wt\rh_j-\wt\rx_j)}\\
		\times J_L(u_1,\cdots,u_m,v_1,\cdots,v_m)\cdot\frac{g_L(u_j;\tilde \tau_j,\tilde\rx_j)}{g_L(v_j;\tilde \tau_j,\tilde\rx_j)}\frac{\d \vvu}{(2\pi \i)^m}\frac{\d\vvv}{(2\pi\i)^m},
\end{multline*}
with $\vec\Sigma_{\rmL/\rmR} :=\Sigma_{1,\rmL/\rmR}\times \cdots\times\Sigma_{m,\rmL/\rmR}$.
Note that $\prod_{j=1}^m(\rf(u_j;\wt \tau_j,\wt\rh_j-\wt\rx_j)\cdot \rf(v_j;\wt \tau_j,\wt\rh_j+\wt\rx_j))$ is uniformly bounded, integrable along the integration contours, and not depending on $L$.
Note that for fixed $u_j$ and $v_j$, $1\le j\le m$, we have
\begin{equation*}
	\begin{split}
	&J_L(u_1,\cdots,u_m,v_1,\cdots,v_m)\to \prod_{j=1}^{m-1}\frac{1}{(u_j-u_{j+1})(v_j-v_{j+1})},\\
	&g_L(u_j;\tilde \tau_j,\tilde\rx_j)\to 1, \qquad g_L(v_j;\tilde \tau_j,\tilde\rx_j)\to 1,
	 \end{split}
\end{equation*}
as $L\to\infty$. 
Moreover, we can see that
\begin{equation}
\label{eq:DCT}
	\sup_L\sup_{\substack{u_j\in \Sigma_{j,\rmL},\\ v_j\in\Sigma_{j,\rmR},\\ j=1,\dots,m}}\abs{	J_L(u_1,\cdots,u_m,v_1,\cdots,v_m)\prod_{j=1}^m\frac{g_L(u_j;\tilde \tau_j,\tilde\rx_j)}{g_L(v_j;\tilde \tau_j,\tilde\rx_j)}}<\infty.
\end{equation}
Indeed, we have $|u_j-u_{j+1}|\ge \dist(\Sigma_{j,\rmL},\Sigma_{j+1,\rmL})=\sqrt 3/2$ and similarly $|v_j-v_{j-1}|\ge \sqrt 3/2$. Moreover, $|1-2^{-3/2}(u_j-v_j)L^{-3/4}|\ge \Re (1-2^{-3/2}(u_j-v_j)L^{-3/4})\ge 1$. Plugging these bounds to~\eqref{eq:J01} and using the simple inequality $|1+a+b|\le (1+|a|)(1+|b|)$, we obtain
\begin{equation}
	\label{eq:J02}
	|J_L (u_1,\cdots,u_m,v_1,\cdots,v_m)|\le C\prod_{j=1}^{m}(1+2^{-3/2}|u_j|L^{-3/4})^2(1+2^{-3/2}|v_j|L^{-3/4})^2
\end{equation}
for some constant $C$ independent of $L$. Finally, we see that $(1+2^{-3/2}|u_j|L^{-3/4})^2 g_L(u_j;\tilde\tau_j,\tilde \rx_j)$ and $(1+2^{-3/2}|v_j|L^{-3/4})^2/ g_L(v_j;\tilde\tau_j,\tilde \rx_j)$ are uniformly bounded due to the super-exponential decay of  $g_L(u_j;\tilde\tau_j,\tilde \rx_j)$ when $u_j\in\Sigma_{j,\rmL}\to\infty$ and $1/g_L(v_j;\tilde\tau_j,\tilde\rx_j)$ when $v_j\in\Sigma_{j,\rmR}\to\infty$. Together with~\eqref{eq:J02} we obntain~\eqref{eq:DCT}.

Now we apply the bounded convergence theorem, and have
\begin{multline}
	\label{eq:limit_Q_step1}
	\lim_{L\to\infty}\frac{\what Q_{\step,\vvx_L,\vvtau}\topp{\vv1}(\vvh_L)}{(8\pi L)^{-1}e^{-\frac{4}{3}L^{3/2}}}\\
	=2\pi (-1)^{m-1}\int_{\vec\Sigma_{\rmL}}\frac{\prod_{j=1}^m\rf(u_j;\wt \tau_j,\wt\rh_j-\wt\rx_j)}{\prod_{j=1}^{m-1}(u_j-u_{j+1})}\frac{\d\vvu}{(2\pi \i)^m}
	\int_{\vec\Sigma_{\rmR}}\frac{\prod_{j=1}^m\rf(v_j;\wt \tau_j,-\wt\rh_j-\wt\rx_j)}{\prod_{j=1}^{m-1}(v_j-v_{j+1})}\frac{\d\vvv}{(2\pi\i)^m}.
\end{multline} 

We remark that now in the expression of the right-hand side of \eqref{eq:limit_Q_step1},  we are free to deform the contours $\Sigma_{j,\rmL}$ and $\Sigma_{j,\rmR}$ on the right-hand side of the above equation to vertical lines as long as the order of the contours are not changed. Therefore, we apply Lemma~\ref{lem:bridge} and obtain
\[
\sqrt{2\pi} \int_{\vec\Sigma_{\rmL}}\frac{\prod_{j=1}^m\rf(u_j;\wt \tau_j,\wt\rh_j-\wt\rx_j)}{\prod_{j=1}^{m-1}(u_j-u_{j+1})}
\frac{\d\vvu}{(2\pi \i)^m}
=\proba\pp{\B^{\rmbr}({\tau_{j}})>\rh_j-\rx_j, j=1,\dots,m-1}.
\]Applying Lemma~\ref{lem:bridge} again but with the indices $j\to m+1-j$ (due to the different order of the contours), we obtain
\begin{align*}
		\sqrt{2\pi}(-1)^{m-1}\int_{\vec\Sigma_{\rmR}} \frac{\prod_{j=1}^m\rf(v_j;\wt \tau_j,-\wt\rh_j-\wt\rx_j)}{\prod_{j=1}^{m-1}(v_j-v_{j+1})}
		\frac{\d\vvv}{(2\pi\i)^m}
		&=\proba\pp{\B^{\rmbr}({1-\tau_{j}})>\rh_j+\rx_j, j=1,\dots,m-1}\\
		&=\proba\pp{\B^{\rmbr}(\tau_{j})>\rh_j+\rx_j, j=1,\dots,m-1},
\end{align*}
where the last step follows by the reversibility of Brownian bridge.
Inserting the above formulas to~\eqref{eq:limit_Q_step1}, we obtain
\begin{align*}
		\lim_{L\to\infty}& \frac{\what Q_{\step,\vvx_L,\vvtau}\topp{\vv1}(\vvh_L)}{(8\pi L)^{-1}e^{-\frac{4}{3}L^{3/2}}}\\
		& =\proba\pp{\B^{\rmbr}({\tau_{j}})>\rh_j-\rx_j, j=1,\dots,m-1}\proba\pp{\B^{\rmbr}(\tau_{j})>\rh_j+\rx_j, j=1,\dots,m-1}\\
		& = \mathsf Q_{\vvrx,\vvtau}(\vvrh),
\end{align*} 
as desired.

\subsection{Proof of Lemma \ref{lem:3}}\label{sec:3.3}
Recall $\what Q\topp\vvn_{\step,\vvx,\vvtau}(\vvh)$ in \eqref{eq:Q_step n} and $\what {\mathsf D}_{\step,\vvx,\vvtau}\topp\vvn$ in \eqref{eq:whatD}. 
For every $\vvn\in\N^m\setminus\{\vv1\}$,  we first control
\begin{align}
|\what{\mathsf D}\topp\vvn_{\step,\vvx_L,\vvtau}(\vvz,\vvh_L)|
& \le \prod_{i_1=1}^{n_1}\int_{C_{1,\rmL}}\frac{|\d\xi_{i_1}\topp 1|}{2\pi}
\prod_{i_1=1}^{n_1}\int_{C_{1,\rmR}}\frac{|\d\eta_{i_1}\topp 1|}{2\pi}
\nonumber\\
& ~\times
  \prodd j2m\Bigg[|1-z_{j-1}|^{n_{j-1}}\abs{1-\frac1{|z_{j-1}|}}^{n_{j}}\nonumber\\
& ~\times \prod_{i_j=1}^{n_j}\pp{\frac1{|1-z_{j-1}|}\int_{C_{j,\rmL}^{\rm in}}\frac{|\d\xi_{i_j}\topp j|}{2\pi}+ \frac{|z_{j-1}|}{|1-z_{j-1}|}\int_{C_{j,\rmL}^{\rm out}}\frac{|\d\xi_{i_j}\topp j|}{2\pi}}\nonumber\\
& ~\times \pp{\frac1{|1-z_{j-1}|}\int_{C_{j,\rmR}^{\rm in}}\frac{|\d\eta_{i_j}\topp j|}{2\pi}+ \frac{|z_{j-1}|}{|1-z_{j-1}|}\int_{C_{j,\rmR}^{\rm out}}\frac{|\d\eta_{i_j}\topp j|}{2\pi}}\Bigg]
\nonumber\\
& ~\times|\Pi_\vvn(\vec{\vv\xi}\topp m_L,\vec{\vv\eta}\topp m_L)|
\times 
\prodd j1m \prodd {i_j}1{n_j}\abs{\frac{f_{\wt x_{L,j},\wt \tau_j}(\xi_{i_j}^{(j)},\wt h_{L,j})}{f_{\wt x_{L,j},\wt \tau_j}(\eta_{i_j}^{(j)},\wt h_{L,j})}}
. \label{eq:3 factors}
\end{align}
Again, the above is a compact way of writing the sum of $2^{n_1+\cdots+n_m}$ number of $(2n_1+\cdots+2n_m)$-multiple integrals with respect to $\xi\topp j_{i_j}, \eta\topp j_{i_j},j=1,\dots,m, i_j=1,\dots,n_j$. 
This time we consider
\[
C_{j,\rmL}^{\rmin/\rmout} = \ccbb{-\sqrt{L}+ 2^{-1/2}L^{-1/4} u:u\in \Sigma_{j,\rmL}^{\rmin/\rmout}}
\]
with each $
\Sigma_{j,\rmL}^{\rmin/\rmout}$ defined as before in \eqref{eq:Sigma}. 
The contours are then disjoint, and the minimal distance among all the pairs is denoted by $\mathfrak c_{\dist} L^{-1/4}$ with $\mathfrak c_{\dist}=\sqrt 3/2$.

As argued in \cite{liu22when}, the first two products in $\Pi_\vvn(\vec{\vv\xi}\topp m_L,\vec{\vv\eta}_L\topp m)$   in \eqref{eq:Pi_n}, where we recall
\[
\vec{\vv\xi}\topp m_L = \pp{\vv\xi\topp 1_L,\dots,\vv\xi\topp m_L} \qmwith \vv\xi\topp j_L = \pp{\xi\topp j_{L,1},\dots,\xi\topp j_{L,n_j}}\in \pp{C_{j,\rmL}^{\rmin/\rmout}}^{n_j}, j=1,\dots,m, 
\] can be re-written as
\begin{multline*}
\prodd j1{m-1}\frac{\Delta(\vv\xi\topp j;\vv\eta\topp{j+1})\Delta(\vv\eta\topp j;\vv\xi\topp {j+1})}{\Delta(\vv\xi\topp j;\vv\xi\topp{j+1})\Delta(\vv\eta\topp j;\vv\eta\topp {j+1})} \prodd j1m \pp{\frac{\Delta(\vv\xi\topp j)\Delta(\vv\eta\topp j)}{\Delta(\vv\xi\topp j;\vv\eta\topp j)}}^2 
= \frac{\Delta(\vv\xi\topp 1)\Delta(\vv\eta\topp 1)}{\Delta(\vv\xi\topp 1;\vv\eta\topp 1)}\\
\times \prodd j1{m-1}\frac{\Delta(\vv\xi\topp j;\vv\eta\topp{j+1})\Delta(\vv\eta\topp j;\vv\xi\topp {j+1})}{\Delta(\vv\xi\topp j;\vv\xi\topp{j+1})\Delta(\vv\eta\topp j;\vv\eta\topp {j+1})} \frac{\Delta(\vv\xi\topp j)\Delta(\vv\eta\topp j)}{\Delta(\vv\xi\topp j;\vv\eta\topp j)}\frac{\Delta(\vv\xi\topp {j+1})\Delta(\vv\eta\topp {j+1})}{\Delta(\vv\xi\topp {j+1};\vv\eta\topp {j+1})}\times \frac{\Delta(\vv\xi\topp m)\Delta(\vv\xi\topp m)}{\Delta(\vv\xi\topp m;\vv\eta\topp m)},
\end{multline*} and each of the three terms can be interpreted as, up to a possible sign depending on $\vvn$, a determinant. Then by Hadamard's inequality, we have (e.g.~compare with $B_1,B_2,B_3$ in  \cite[Section 3.1]{liu21one}),
\[
\abs{\frac{\Delta(\vv\xi\topp 1)\Delta(\vv\eta\topp 1)}{\Delta(\vv\xi\topp 1;\vv\eta\topp 1)}} \le n_1^{n_1/2}\pp{\frac1{\mathfrak c_\dist L^{-1/4}}}^{n_1},\quad
\abs{\frac{\Delta(\vv\xi\topp m)\Delta(\vv\eta\topp m)}{\Delta(\vv\xi\topp m;\vv\eta\topp m)}} \le n_m^{n_m/2}\pp{\frac1{\mathfrak c_\dist L^{-1/4}}}^{n_m},
\]
and
\begin{multline*}
\abs{\frac{\Delta(\vv\xi\topp j;\vv\eta\topp{j+1})\Delta(\vv\eta\topp j;\vv\xi\topp {j+1})}{\Delta(\vv\xi\topp j;\vv\xi\topp{j+1})\Delta(\vv\eta\topp j;\vv\eta\topp {j+1})} \frac{\Delta(\vv\xi\topp j)\Delta(\vv\eta\topp j)}{\Delta(\vv\xi\topp j;\vv\eta\topp j)}\frac{\Delta(\vv\xi\topp {j+1})\Delta(\vv\eta\topp {j+1})}{\Delta(\vv\xi\topp {j+1};\vv\eta\topp {j+1})}}\\
 \le (n_j+n_{j+1})^{(n_j+n_{j+1})/2} \pp{\frac1{\mathfrak c_\dist L^{-1/4}}}^{n_j+n_{j+1}}, j=1,\dots,m-1.
\end{multline*}
We also bound 
$\sabs{\sum_{i_m=1}^{n_m}(\xi_{L,i_m}\topp m- \eta_{L,i_m}\topp m)} \le \prod_{i_m=1}^{n_m}\spp{(1+|\xi_{L,i_m}\topp m|)(1+|\eta_{L,i_m}\topp m|)}$.  
Therefore,
\begin{multline}\label{eq:L}
|\Pi_\vvn(\vec{\vv\xi}\topp m_L,\vec{\vv\eta}\topp m_L)| \le n_1^{\frac{n_1}2}\pp{\prodd j1{m-1}(n_j+n_{j+1})^{\frac{n_j+n_{j+1}}2}}n_m^{\frac{n_m}2} \cdot (\mathfrak c_\dist^{-2}L^{1/2})^{n_1+\cdots+n_m}\\
\times \prod_{i_m=1}^{n_m}\pp{(1+|\xi_{L,i_m}\topp m|)(1+|\eta_{L,i_m}\topp m|)}.
\end{multline}

Next, we examine  the last term in \eqref{eq:3 factors}. We eventually shall integral over $\Sigma_{j,\rmL/\rmR}$ and we view 
\[
\xi_{L,i_j}\topp j  =\xi_{L,i_j}\topp j (u\topp j_{i_j}) = \sqrt L + 2^{-1/2}L^{-1/4} u_{i_j}\topp j \quad\mbox{ for } \quad u\topp j_{i_j}\in \Sigma_{j,\rmL}^{\rmin/\rmout}, \quad j=1,\dots,m, i_j=1,\dots,n_j,
\] as a function of $u\topp j_{i_j}$. Recalling the analysis of $f_{\wt x_{L,j},\wt\tau_j}$ in \eqref{eq:f}, we have
\begin{align*}
\prodd {i_j}1{n_j}f_{\wt x_{L,j},\wt \tau_j}(\xi_{L,i_j}^{(j)},\wt h_{L,j}) & = \exp\pp{\summ j1m\sum_{i_j=1}^{n_j}\pp{-\frac23 \wt\tau_j L^{3/2} + \pp{\frac{\wt\rx_j}{\sqrt 2} - \sqrt 2\wt\rh_j}L^{3/4}}}\\
& \quad\times \prodd j1m\prod_{i_j=1}^{n_j} \rf \pp{u\topp j_{i_j};\wt\tau_j,\wt\rh_j-\wt\rx_j}g_L\pp{u\topp j_{i_j};\wt\tau_j-\wt\rx_j}.
\end{align*}
We see that for every $\epsilon>0$, one can take $L$ large enough so that
\[
-\frac23 \wt\tau_j L^{3/2} + \pp{\frac{\wt\rx_j}{\sqrt 2} - \sqrt 2\wt\rh_j}L^{3/4} \le -\frac{2(1-\epsilon)}3 \wt\tau_j L^{3/2}, j=1,\dots, m.
\]
We shall assume the above holds in the sequel. We have also seen that $\sup_{L>0, u_j\in\Sigma_{j,\rmL}^{\rmin/\rmout}}|g_L(u\topp j_{i_j};\wt\tau_j-\wt\rx_j)|<\infty$ (see \eqref{eq:DCT}). Similar analysis applies to $\prodd {i_j}1{n_j}f_{\wt x_{L,j},\wt \tau_j}(\xi_{L,i_j}^{(j)},\wt h_{L,j})$.
We have, for every $\epsilon>0$ fixed and then for $L$ large enough,
\begin{align}\label{eq:D hat n}
|\what{\mathsf D}\topp\vvn_{\step,\vvx_L,\vvtau}(\vvz,\vvh_L)|
& \le 
C\exp\pp{-\frac {4(1-\epsilon)}3 \summ j1m \wt\tau_jn_j L^{3/2}}
\\
& \quad\times  \prodd j2m|1-z_{j-1}|^{n_{j-1}}\abs{1-\frac{1}{z_{j-1}}}^{n_{j}}\nonumber\\
& \quad\times
n_1^{\frac{n_1}2}\pp{\prodd j1{m-1}(n_j+n_{j+1})^{\frac{n_j+n_{j+1}}2}}n_m^{\frac{n_m}2} \cdot  C^{n_1+\cdots+n_m} \times \mathfrak C_n(L)\mathfrak D_n(L),\nonumber
\end{align}
with
\begin{align*}
\mathfrak C_n(L) & = \prod_{i_1=1}^{n_1}\int_{\Sigma_{1,\rmL}}\abs{\rf\pp{u_{i_1}\topp 1;\wt\tau_1,\wt\rh_1-\wt\rx_1}}\d |u\topp 1_{i_1}|\\
& \quad\times \prodd j2{m-1}\prod_{i_j=1}^{n_j}
\int_{\Sigma_{j,\rmL}^\rmin \cup \Sigma_{j,\rmL}^\rmout}\abs{\rf\pp{u_{i_j}\topp j;\wt\tau_j,\wt\rh_j-\wt\rx_j}}\d |u\topp j_{i_j}|\nonumber\\
& \quad\times \prod_{i_m=1}^{n_m}
\int_{\Sigma_{m,\rmL}^\rmin \cup \Sigma_{m,\rmL}^\rmout}\abs{\rf\pp{u_{i_m}\topp m;\wt\tau_m,\wt\rh_m-\wt\rx_m}}
\pp{1+|\xi_{L,i_m}\topp m(u_{i_m}\topp m)|}\d |u\topp m_{i_m}|,\nonumber
\end{align*}
and a similar expression for $\mathfrak D_n(L)$. 
Note that we have canceled the factor $L^{(n_1+\cdots+n_m)/2}$ from \eqref{eq:L} and those from $\d\xi_{i_j}\topp j = 2^{-1/2}L^{-1/4}\d u\topp j_{i_j}$ and $\d\eta_{i_j}\topp j = 2^{-1/2}L^{-1/4}\d v\topp j_{i_j}$. 
Then,
\[
\mathfrak C_n(L)\le C^{n_1+\cdots+n_m} L^{n_m/2}.
\]
 The same upper bound for $\mathfrak D_n(L)$ holds.

Therefore, for every $\epsilon>0$, there exists $C>0$ such that for $L$ large enough, 
\begin{multline*}
\bigg|\oint_{>1} \cdots \oint_{>1}  
\what{\mathsf D}\topp\vvn_{\step,\vvx_L,\vvtau}(\vvz,\vvh_L)
\frac{\d z_1}{2\pi\i z_1(1-z_1)}\cdots\frac{\d z_{m-1}}{2\pi \i z_{m-1}(1-z_{m-1})}\bigg|
\\
\le
n_1^{\frac{n_1}2}\pp{\prodd j1{m-1}(n_j+n_{j+1})^{\frac{n_j+n_{j+1}}2}}n_m^{\frac{n_m}2}  \times C^{n_1+\dots+n_m} \times  L^{n_m/2}\exp\pp{-\frac {4(1-\epsilon)}3 \summ j1m \wt\tau_jn_j L^{3/2}},
\end{multline*}
for all $\vvn\in\N^m\setminus\{\vv1\}$.
We also have, for $L$ large enough,
\begin{align*}
L^{n_m/2}\exp\pp{-\frac {4(1-\epsilon)}3 \summ j1m \wt\tau_jn_jL^{3/2}} & \le \exp\pp{-\frac {4(1-2\epsilon)}3 \summ j1m \wt\tau_jn_jL^{3/2}}\\
&  \le \exp\pp{-\frac {4(1-2\epsilon)}3 \pp{1+\min_{j=1,\dots,m}\wt \tau_j}L^{3/2}}.
\end{align*}
We write $1+\epsilon' \equiv (1-2\epsilon)(1+\min_{j=1,\dots,m}\wt\tau_j)$. Notice that in this way, by taking $\epsilon>0$ small enough, we have $\epsilon'\in(0,\min_{j=1,\dots,m}\wt \tau_j)$ as desired.

It remains to show that the factors involving $n_1,\dots,n_m$ above are summable. For $n_1\le \cdots\le n_m$ we have
\begin{align*}
n_1^{\frac{n_1}2}\pp{\prodd j1{m-1}(n_j+n_{j+1})^{\frac{n_j+n_{j+1}}2}}n_m^{\frac{n_m}2}
& \le 2^{n_2+\cdots+n_m}n_1^{n_1/2}n_2^{n_2}\cdots n_{m-1}^{n_{m-1}} n_m^{3n_m/2}.
\end{align*}
Therefore, by Stirling's formula, with $\vvn! = n_1!\cdots n_m!$,
\begin{multline*}
\sum_{\vvn\in\N^m\setminus\{\vv1\}}  \frac{n_1^{\frac{n_1}2}\pp{\prodd j1{m-1}(n_j+n_{j+1})^{\frac{n_j+n_{j+1}}2}}n_m^{\frac{n_m}2}}{(\vvn!)^2}  \times  C^{n_1+\dots+n_m} \\
 \le Cm!\sum_{1\le n_1\le\cdots\le n_m} C^{n_1+\cdots+n_m}\frac{n_m^{n_m/2}}{\vvn!}<\infty.
\end{multline*}
Combining the above we have proved the desired estimate.

\subsection{Joint convergence in general: with time points not necessarily distinct}\label{sec:ZP}
The following general fact is of its own interest. This section does not require specific laws of the models involved earlier. Let, for each $n\in\N$, $\{Y_n(x,t)\}_{x\in \R, t\in (0,T)}$ be a random field, and we are interested in the convergence of finite-dimensional distributions to another random field $\{Y(x,t)\}_{x\in\R,t\in(0,T)}$, as $n\to\infty$. 
\begin{Lem}\label{lem:bootstrap}
Assume that 
\[
\pp{Y_n(x_i,t_i)}_{i=1,\dots,d}\weakto \ccbb{Y(x_i,t_i)}_{i=1,\dots,d}
\]
for all $d\in\N$, $x_i\in\R,t_i\in(0,T), i=1,\dots,d$, such that all $t_1,\dots,t_d$ are distinct. Assume also that the limit random field $\{Y(x,t)\}_{x\in\R,t\in(0,T)}$ has continuous joint cumulative distribution functions.  Then, $\ccbb{Y_n(x,t)}_{x\in \R,t\in (0,T)}\fddto \ccbb{Y(x,t)}_{x\in \R,t\in(0,T)}$. 
\end{Lem}
\begin{proof}
We first prove the case $d=2$. We shall prove
\[
\limn \proba\pp{Y_n(x_1,t)> y_1,Y_n(x_2,t)> y_2}  = \proba\pp{Y(x_1,t)> y_1,Y(x_2,t)> y_2}, 
\]
for all $x_1,x_2, y_1,y_2\in\R, t\in(0,T)$. 
First, we write, for $\epsilon\in\R$ such that $t+\epsilon\in(0,T)$ and $\delta>0$,
\begin{align*}
\proba & \pp{Y_n(x_1,t)> y_1,Y_n(x_2,t)> y_2}  
\\
& \ge \proba\pp{Y_n(x_1,t)> y_1,Y_n(x_2,t)> y_2,Y_n(x_2,t+\epsilon)> y_2+\delta}  \\
&  = \proba\pp{Y_n(x_1,t)> y_1,Y_n(x_2,t+\epsilon)> y_2+\delta}  \\
& \quad - \proba\pp{Y_n(x_1,t)> y_1,Y_n(x_2,t)\le y_2,Y_n(x_2,t+\epsilon)> y_2+\delta} \\
& \ge \proba\pp{Y_n(x_1,t)> y_1,Y_n(x_2,t+\epsilon)> y_2+\delta}   - \proba\pp{Y_n(x_2,t)\le y_2,Y_n(x_2,t+\epsilon)> y_2+\delta}.
\end{align*}
Note that in the last expression above, each probability concerns the joint law of random field at distinct time points. Therefore, we have
\begin{multline*}
\liminf_{n\to\infty}\proba  \pp{Y_n(x_1,t)> y_1,Y_n(x_2,t)> y_2}  \\
\ge 
\proba\pp{Y(x_1,t)> y_1,Y(x_2,t+\epsilon)> y_2+\delta}   - \proba\pp{Y(x_2,t)\le y_2,Y(x_2,t+\epsilon)> y_2+\delta}.
\end{multline*}
Letting $\epsilon,\delta\downarrow 0$, by continuity of the joint law  we have
\equh\label{eq:liminf}
\liminf_{n\to\infty}\proba  \pp{Y_n(x_1,t)> y_1,Y_n(x_2,t)> y_2}  \ge 
\proba\pp{Y(x_1,t)> y_1,Y(x_2,t)> y_2}.
\eque
For the other direction, write
\begin{align*}
\proba & \pp{Y_n(x_1,t)> y_1,Y_n(x_2,t)> y_2}  \\
&  = \proba\pp{Y_n(x_1,t)> y_1,Y_n(x_2,t)> y_2, Y_n(x_2,t+\epsilon)> y_2-\delta}  \\
& \quad + \proba\pp{Y_n(x_1,t)> y_1,Y_n(x_2,t)>  y_2,Y_n(x_2,t+\epsilon)\le y_2-\delta} \\
& \le \proba\pp{Y_n(x_1,t)> y_1,Y_n(x_2,t+\epsilon)> y_2-\delta}   + \proba\pp{Y_n(x_2,t)> y_2,Y_n(x_2,t+\epsilon)\le y_2-\delta}.
\end{align*}
Again, first taking $\limsup$ as $n\to\infty$, and then letting $\epsilon,\delta\downarrow0$, we have
\[
\limsup_{n\to\infty}\proba  \pp{Y_n(x_1,t)> y_1,Y_n(x_2,t)> y_2}  \le
\proba\pp{Y(x_1,t)> y_1,Y(x_2,t)> y_2}.
\]
Combining with \eqref{eq:liminf}, we have proved the convergence for $d=2$. For larger $d\in\N$, the proof can be carried out by the same method and induction on the number of $t_i$ that take the same values. The details are omitted.
\end{proof}

\section{Proof for the case with flat initial condition}\label{sec:flat}
Assume that $Z$ is 
a standard normal random variable, and let
 $\B^{\rmbr, Z/\sqrt 2}_1$ be a Brownian bridge with initial value $Z/\sqrt 2$ at $t=0$ and value 0 at $t = 1$, and $\B^{\rmbr,Z/\sqrt 2}_2$ be another process defined similarly, and the two are conditionally independent given $Z$. Equivalently, the two processes can be defined via
 \[
\B^{\rmbr,Z/\sqrt 2}_i(\tau):=\B^{\rmbr}_i(\tau) + \frac{1-\tau}{\sqrt 2} Z, \quad \tau\in[0,1], i=1,2,
 \]
 where $\B^\rmbr_1,\B^\rmbr_2$ are two i.i.d.~standard Brownian bridge, independent from $Z$.
We shall prove the following restatement of Theorem \ref{thm:2} in the case with flat initial condition:
\begin{multline*}
\calL\pp{\ccbb{\frac{\H_\flat(\frac{\rx}{\sqrt 2L^{1/4}},\tau)-\tau\H_\flat(0,1)}{\sqrt 2L^{1/4}} }_{\rx\in\R,\tau\in(0,1)}\mmid  \H_\flat(0,1) = L}\\
 \fddto \calL\pp{\ccbb{\min\ccbb{\B^{\rmbr,Z/\sqrt 2}_1(\tau)+\rx,\B^{\rmbr,-Z/\sqrt 2}_2(\tau)-\rx}}_{\rx\in\R,\tau\in(0,1)}},
\end{multline*}
as $L\to\infty$.

The proof follows the same strategy as in the case of step initial condition. So we only sketch the key calculations. 
We use the same notation for $\vv{\rx},\vv\tau, \vv{\rh}, \vvx_L, \vvh_L$ as in \eqref{eq:h_L,j}.  
Again we prove the case $\tau_0,\dots,\tau_m$ are all distinct and then apply Lemma \ref{lem:bootstrap}. 
This time,
we consider
\[
\proba\pp{\H_\flat(x_{L,j},\tau_j)>h_{L,j}, j=1,\dots,m-1 \mmid \H_\flat(0,1) = L} = \frac{\what Q_{\flat,\vvx_L,\vvtau}(\vvh_L)}{p_{\flat,0,1}(L)}.
\]
The derivation of the expression  of $\what Q_{\flat,\vvx,\vvtau}$ is similar to the one of \eqref{eq:Q conditional} as before following \cite[Proposition 2.9 and Definition 2.26]{liu22multipoint}. More precisely, we have
\[
\what Q_{\flat,\vvx,\vvtau}(\vvh)  :=\proba\pp{\H_\flat(x_j,\tau_j)>h_j, j=1,\dots,m-1, \H_\flat(x_m,\tau_m) = h_m}
 = \sum_{\vvn\in\N_0^m}\frac1{(\vvn!)^2}\what Q\topp\vvn_{\flat,\vvx,\vvtau}(\vvh),
\]
where
\[
\what Q_{\flat,\vvx,\vvtau}\topp\vvn(\vvh) := (-1)^{m-1}\oint_{>1} \cdots \oint_{>1} {\what{\mathsf D}\topp\vvn}_{\flat,\vvx,\vvtau}(\vvz,\vvh)\frac{\d z_1}{2\pi\i z_1(1-z_1)}\cdots\frac{\d z_{m-1}}{2\pi \i z_{m-1}(1-z_{m-1})}, 
\]
with
\begin{align*}
\what{\mathsf D}\topp\vvn_{\flat,\vvx,\vvtau}(\vvz,\vvh) 
&   := \prod_{i_1=1}^{n_1} \int_{C_{1,\rmL}}\frac{\d\xi_{i_1}\topp 1}{2\pi\i}
\int_{C_{1,\rmR}}\frac{\d\eta_{i_1}\topp 1}{2\pi\i}
\det\pp{\ccbb{\delta(-\eta_k\topp 1,\xi_\ell\topp 1)}_{k,\ell=1,\dots,n_1}}
\nonumber\\
& \quad\times  \prodd j2m\Bigg[(1-z_{j-1})^{n_{j-1}}\pp{1-\frac1{z_{j-1}}}^{n_{j}}\\
& \quad\times \prod_{i_j=1}^{n_j}\pp{\frac1{1-z_{j-1}}\int_{C_{j,\rmL}^{\rm in}}\frac{\d\xi_{i_j}\topp j}{2\pi\i}- \frac{z_{j-1}}{1-z_{j-1}}\int_{C_{j,\rmL}^{\rm out}}\frac{\d\xi_{i_j}\topp j}{2\pi\i}}\nonumber\\
& \quad\times \pp{\frac1{1-z_{j-1}}\int_{C_{j,\rmR}^{\rm in}}\frac{\d\eta_{i_j}\topp j}{2\pi\i}- \frac{z_{j-1}}{1-z_{j-1}}\int_{C_{j,\rmR}^{\rm out}}\frac{\d\eta_{i_j}\topp j}{2\pi\i}}\Bigg]\nonumber\\
&\quad\times\wt\Pi_\vvn(\vec{\vv\xi}\topp m,\vec{\vv\eta}\topp m) \prodd j1m \prod_{i_j=1}^{n_j}\frac{f_{\wt x_j,\wt \tau_j}(\xi_{i_j}\topp j,\wt h_j)}{f_{\wt x_j,\wt\tau_j}(\eta_{i_j}\topp j,\wt h_j)},\nonumber
\end{align*}
with
\[
\wt\Pi_\vvn(\vec{\vv\xi}\topp m,\vec{\vv\eta}\topp m) := \Pi_\vvn(\vec{\vv\xi}\topp m,\vec{\vv\eta}\topp m) \times (-1)^{n_1(n_1+1)/2} \frac{\Delta(\vv\xi\topp1;\vv\eta\topp1)}{\Delta(\vv\xi\topp1)\Delta(\vv\eta\topp1)}.
\]
In particular, the same factor $\Pi_\vvn$ in the previous section appears again, and the only difference is the second factor on the right-hand side above.
Moreover, the above multiple integrals assume the additional assumption that
\[
C_{1,\rmL} = -C_{1,\rmR}
\]
(equality as two sets), and the $\delta$ function is such that
\equh\label{eq:delta}
\int_{C_{1,\rmL}}\delta(-\eta,\xi)f(\xi)\frac{\d\xi}{2\pi\i} = f(-\eta),\quad \mfa f\in L^2\pp{C_{1,\rmL},\frac{\d\xi}{2\pi\i}}, \eta\in C_{1,\rmR}.
\eque
Throughout, we follow the same notations as in Section \ref{sec:step}.
Again, we have similarly as before the following:
\begin{enumerate}[(i)]
\item the term $\what Q\topp{\vv1}_{\flat,\vvx_L,\vvtau}(\vvh_L)$ has the desired asymptotic behavior.
\item when $\vvn \in\N_0^m\setminus\N^m$, $\what Q\topp{\vv1}_{\flat,\vvx,\vvtau}(\vvh) = 0$.
\item the remainder is negligible: 
\equh\label{eq:flat remainder}
\sum_{\vvn\in\N^p\setminus\{\vv1\}}\what Q\topp{\vvn}_{\flat,\vvx_L,\vvtau}(\vvh_L) = o\pp{\what Q\topp{\vv1}_{\flat,\vvx_L,\vvtau}(\vvh_L)}.
\eque
\end{enumerate}
The proof for the second fact above is the same as in the proof of Lemma \ref{lem:1}. The proof for \eqref{eq:flat remainder} follows closely the proof of Lemma \ref{lem:3}, and the difference is on the estimates of $\wt\Pi_\vvn$ and also on the estimate of $|\what{\mathsf D}\topp\vvn_{\flat,\vvx_L,\vvtau}(\vvz,\vvh_L)|$ in place of  \eqref{eq:D hat n}, of which the details are omitted. So we arrive at
\[
\frac{\what Q_{\flat,\vvx_L,\vvtau}(\vvh_L)}{p_{\flat}(L)} \sim \frac{\what Q\topp{\vv1}_{\flat,\vvx_L,\vvtau}(\vvh_L)}{p_{\flat}(L)} \mmas L\to\infty,
\]
and we only focus on the  contributing term. 
By \eqref{eq:delta},
\begin{align*}
\what{\mathsf D}\topp{\vv1}_{\flat,\vvx,\vv\tau}(\vvz,\vvh) & =  \int_{C_{1,\rmR}}\frac{\d\eta\topp1}{2\pi\i}\times \prodd j2{m}\Bigg[(1-z_{j-1})\pp{1-\frac1{z_{j-1}}}\\
& \quad\times \pp{\frac1{1-z_{j-1}}\int_{C_{j,\rmL}^\rmin}\frac{\d \xi\topp j}{2\pi\i} - \frac{z_{j-1}}{1-z_{j-1}}\int_{C_{j,\rmL}^\rmout}\frac{\d\xi\topp j}{2\pi \i}}\\
& \quad\times \pp{\frac1{1-z_{j-1}}\int_{C_{j,\rmR}^\rmin}\frac{\d \eta\topp j}{2\pi\i} - \frac{z_{j-1}}{1-z_{j-1}}\int_{C_{j,\rmR}^\rmout}\frac{\d\eta\topp j}{2\pi \i}}\Bigg]\\
& \quad\times \wt \Pi_{\vv1}(\vec\xi\topp m,\vec\eta\topp m)
 \prodd j1m\frac{f_{\wt x_j,\wt \tau_j}(\xi\topp j,\wt h_j)}{f_{\wt x_,\wt\tau_j}(\eta\topp j,\wt h_j)},
\end{align*}
with $\vec\xi\topp m = (\xi\topp 1,\dots,\xi\topp m)$ and $\vec\eta\topp m = (\eta\topp 1,\dots,\eta\topp m)$,
\[
\wt \Pi_{\vv1}(\vec\xi\topp m,\vec\eta\topp m) = \Pi_{\vv1}(\vec\xi\topp m,\vec\eta\topp m) \times(\eta\topp 1- \xi\topp 1),
\]
and the convention 
\[
\xi\topp 1 = -\eta\topp 1.
\]
We have, in place of \eqref{eq:Q_hat_11} for the case with step initial condition,  by the same derivation,
\begin{align*}
\what Q\topp{\vv1}_{\flat,\vvx_L,\vvtau}(\vvh_L) & =  (-1)^{m-1}\oint_{>1} \cdots \oint_{>1} {\what{\mathsf D}\topp{\vv1}}_{\flat,\vvx,\vvtau}(\vvz,\vvh)\frac{\d z_1}{2\pi\i\cdot z_1(1-z_1)}\cdots\frac{\d z_{m-1}}{2\pi \i \cdot z_{m-1}(1-z_{m-1})}\\
& = 
\int_{\vec C_{L,\rmL}^*\times \vec C_{L,\rmR}}
(-1)^{m-1} \wt \Pi_{\vv1}(\vec\xi\topp m,\vec\eta\topp m)
 \prodd j1m\frac{f_{\wt x_{L,j},\wt \tau_j}(\xi\topp j,\wt h_{L,j})}{f_{\wt x_{L,j},\wt\tau_j}(\eta\topp j,\wt h_{L,j})}
\frac{\d\vec\xi^{(m),*}\d\vec\eta\topp m}{(2\pi \i)^{2m-1}},
\end{align*}
with
\begin{align*}
\vec C_{L,\rmL}^* & :=  C_{2,L,\rmL}^\rmout\times\cdots\times C_{m,L,\rmL}^\rmout, \\
\vec C_{L,\rmR} & := C_{1,L,\rmR}\times C_{2,L,\rmR}^\rmout\times\cdots \times C_{m,L,\rmR}^\rmout,
\end{align*}
and $\vec\xi^{(m),*} = (\xi\topp2,\dots,\xi\topp m)$.
We shall work with 
\[
\xi_L\topp j = -\sqrt L + u_j\frac1{\sqrt 2L^{1/4}} \qmand \eta_L\topp j = \sqrt L+v_j\frac1{\sqrt 2L^{1/4}},
\] as before as functions of $u_j, v_j$, respectively.
Write $\vec\Sigma_\rmL^* = \Sigma_{2,\rmL}\times\cdots\times \Sigma_{m,\rmL}$ and $\d\vv u^* = \d u_2\cdots \d u_m$ accordingly. When writing $u_1$ in the sequel we follow the convention $u_1 = -v_1$,  so that several key calculations in the case of step initial condition can be borrowed directly. 

In place of \eqref{eq:Pi_1_rewrite} we have
\begin{align*}
\wt\Pi_1(\vec\xi_L\topp m,\vec\eta_L\topp m)  & = 2^{m-2}L^{m/2-1}   J_L(u_1,\dots,u_m,v_1,\dots,v_m) \times\pp{2\sqrt L+2^{1/2}L^{-1/4}v_1}\\
& \sim 2^{m-1}L^{(m-1)/2}J_L(u_1,\dots,u_m,v_1,\dots,v_m),
\end{align*}
the same formulas and hence asymptotics for $f$ as in \eqref{eq:prod_f_xi}, \eqref{eq:prod_f_eta} hold. This time, $\d\vec\xi^{(m),*}\d\vec\eta\topp m = 2^{-(2m-1)/2}L^{-(2m-1)/4} \d\vvu^*\d\vvv$. Then, in place of \eqref{eq:limit_Q_step1} we arrive at
\begin{align}
\frac{\what Q\topp{\vv1}_{\flat,\vvx_L,\vvtau}(\vvh_L)}{p_{\flat,0,1}(L)} & \sim\frac{\what Q\topp{\vv1}_{\flat,\vvx_L,\vvtau}(\vvh_L)}{(8\pi\sqrt L)^{-1/2}\exp(-\frac 43L^{3/2})} \nonumber\\
& \to 
\sqrt{4\pi}\int_{\vec \Sigma_{\rmL}^*\times \vec \Sigma_\rmR}
(-1)^{m-1}\frac{\prodd j1m \rf(u_j;\wt\tau_j,\wt \rh_j-\wt\rx_j)\rf(v_j;\wt\tau_j,-\wt\rh_j-\wt\rx_j)}{\prodd j1{m-1}(u_j-u_{j+1})(v_j-v_{j+1})}\frac{\d\vvu^*\d\vvv}{(2\pi\i)^{2m-1}},\label{eq:integrand}
\end{align}
as $L\to\infty$. 
Notice that the factorization into two $m$-multiple integrals as in \eqref{eq:limit_Q_step1} no longer holds here. Some extra work is needed. 

To arrive at the desired formula, we first write the integrand of \eqref{eq:integrand} as, for $u_1,\dots,u_m, v_1,\dots,v_m$ fixed,
\begin{align*}
& (-1)^{m-1} \frac{\prodd j1m \rf(u_j;\wt\tau_j,\wt \rh_j-\wt\rx_j)\rf(v_j;\wt\tau_j,-\wt\rh_j-\wt\rx_j)}{\prodd j1{m-1}(u_j-u_{j+1})(v_j-v_{j+1})}\\
& = \frac{\prodd j1m \rf(u_j;\wt\tau_j,0)\rf(v_j;\wt\tau_j,0)}{\prodd j1{m-1}(u_{j+1}-u_{j})(v_j-v_{j+1})} \prodd j1{m-1}\exp\pp{(\rh_j -\rx_j)(-u_{j+1}+u_j)+(-\rh_j-\rx_j)(-v_{j+1}+v_j)}\nonumber\\
& = \int_{\substack{b_j\ge \rh_j-\x_j\\
c_j\le -\rh_j-\rx_j\\
j=1,\dots,m-1}}
\prodd j1m \rf(u_j;\wt\tau_j,0)\rf(v_j;\wt\tau_j,0) \prodd j1{m-1}\exp\pp{b_j(-u_{j+1}+u_j)+c_j(-v_{j+1}+v_j)}\d\vvb\d\vvc\nonumber\\
&  = \int_{\substack{b_j\ge \rh_j-\x_j\\
c_j\le -\rh_j-\rx_j\\
j=1,\dots,m-1}}
\prodd j1m \rf(u_j;\wt\tau_j,\wt b_j)\rf(v_j;\wt\tau_j,\wt c_j) \d\vvb\d\vvc,
\end{align*}
where $\wt b_j = b_j-b_{j-1}, \wt c_j = c_j-c_{j-1}, j=1,\dots,m$ and $b_0 =c_0 =b_m = c_m = 0$  as usual, and  in the first equality we used the fact that ${\rmRe}(u_j-u_{j+1})<0$ and ${\rmRe}(v_j-v_{j+1})>0$.
Next, we recognize
\begin{align}
\int_{\vec\Sigma_\rmL^*\times\vec\Sigma_\rmR}&
\prodd j1m \rf(u_j;\wt\tau_j,\wt b_j)\rf(v_j;\wt\tau_j,\wt c_j) \frac{\d\vvu^*\d\vvv}{(2\pi\i)^{2m-1}}\nonumber\\
& = \int_{\Sigma_{1,\rmR}}\rf(-v_1;\tau_1,b_1)\rf(v_1;\tau_1,c_1)\frac{\d\vvv_1}{2\pi\i} \nonumber\\
& \quad \times\prodd j2m \int_{\Sigma_{j,\rmL}}\rf(u_j;\wt\tau_j,\wt b_j)\frac{\d u_j}{2\pi\i}\times \prodd j2m\int_{\Sigma_{j,\rmR}}\rf(v_j;\wt\tau_j,\wt c_j)\frac{\d v_j}{2\pi\i}\nonumber\\
& = \phi_{2\tau_1}(c_1-b_1)\times \prodd j2m \phi_{\wt\tau_j}(\wt b_j)\times \prodd j2m \phi_{\wt\tau_j}(\wt c_j),\label{eq:3 phi}
\end{align}
where in the last step  for each integral the identity \eqref{eq:aux03}. Also, by the semigroup property $\phi_{2\tau_1}(c_1-b_1) = \int_\R\phi_{\tau_1}(z-b_1)\phi_{\tau_1}(c_1-z)\d z$, \eqref{eq:3 phi} becomes
\begin{align*}
\int_\R &  \phi_{\tau_1}(b_1-z)\prodd j2m \phi_{\wt\tau_j}(\wt b_j) \times \phi_{\tau_1}(c_1-z)\prodd j2m \phi_{\wt\tau_j}(\wt c_j)\d z\\
& = \int_\R\phi_1^2(z) \mathsf p^{\rmbr,z}_{\tau_1,\dots,\tau_{m-1}}(b_1,\dots,b_{m-1})\mathsf p^{\rmbr,z}_{\tau_1,\dots,\tau_{m-1}}(c_1,\dots,c_{m-1})\d z\\
& = \frac1{\sqrt{4\pi}}\int_\R\phi_{1/2}(z) \mathsf p^{\rmbr,z}_{\tau_1,\dots,\tau_{m-1}}(b_1,\dots,b_{m-1})\mathsf p^{\rmbr,z}_{\tau_1,\dots,\tau_{m-1}}(c_1,\dots,c_{m-1})\d z,
\end{align*}
where $\mathsf p^{\rmbr,z}_{\tau_1,\dots,\tau_{m-1}}$ is the conditional joint probability density function of $\B^{\rmbr,Z}$, given $Z = z$, at time points $\tau_1,\dots,\tau_{m-1}$. 

Now, combining all the above, we have shown
\begin{align*}
\lim_{L\to\infty}\frac{\what Q\topp{\vv1}_{\flat,\vvx_L,\vvtau}(\vvh_L)}{p_{\flat}(L)} 
& =
\sqrt{4\pi}\int_{\vec \Sigma_{\rmL}^*\times \vec \Sigma_\rmR}(-1)^{m-1}\frac{\prodd j1m f(u_j;\wt\tau_j,\wt \rh_j-\wt\rx_j)f(v_j;\wt\tau_j,-\wt\rh_j-\wt\rx_j)}{\prodd j1{m-1}(u_j-u_{j+1})(v_j-v_{j+1})}\frac{\d\vvu^*\d\vvv}{(2\pi\i)^{2m-1}}
\\
& = \sqrt{4\pi}\int_{\substack{b_j\ge \rh_j-\x_j\\
c_j\le -\rh_j-\rx_j\\
j=1,\dots,m-1}}\int_{\vec\Sigma_{\rmL}^*\times\vec\Sigma_\rmR}
\prodd j1m \rf(u_j;\wt\tau_j,\wt b_j)\rf(v_j;\wt\tau_j,\wt c_j) \frac{\d\vvu^*\d\vvv}{(2\pi\i)^{2m-1}}\d\vvb\d\vvc\\
& = \proba\pp{\B_1^{\rmbr,Z/\sqrt 2}(\tau_j)\ge  \rh_j-\rx_j, \B_2^{\rmbr,Z/\sqrt 2}(\tau_j)\le -\rh_j-\rx_j, j=1,\dots,m-1}\\\
& = \proba\pp{\B_1^{\rmbr,Z/\sqrt 2}(\tau_j)\ge \rh_j-\rx_j, \B_2^{\rmbr,-Z/\sqrt 2}(\tau_j)\ge \rh_j+\rx_j, j=1,\dots,m-1},
\end{align*}
where the last step follows by symmetry ($-\B^{\rmbr,Z/\sqrt 2}$ has the same law as $\B^{\rmbr,-Z/\sqrt 2}$).



\end{document}